\documentclass[12pt,paper=a4, headsepline,headings=normal,abstract=true]{scrartcl} 
\usepackage[english]{babel}
\usepackage[utf8]{inputenc}
\usepackage[T1]{fontenc} 
\usepackage{textcomp}
\usepackage[margin=2cm]{geometry}
\usepackage[reqno]{amsmath}   
\usepackage{amssymb}          
\usepackage{amsthm}           
\usepackage{amstext}          
\usepackage{enumerate}        
\usepackage{bbm}
\usepackage{pifont}           
\usepackage{framed}
\usepackage{scrpage}
\usepackage{graphicx}
\usepackage{verbatim}
\usepackage{mathrsfs}
\usepackage{url}
\urldef{\urluni}{\url}{http://www.mathematik.uni-kl.de/~wwwfktn}
\urldef{\emailfattler}{\url}{fattler@mathematik.uni-kl.de}
\urldef{\emailgrothaus}{\url}{grothaus@mathematik.uni-kl.de}
\urldef{\emailvosshall}{\url}{vosshall@mathematik.uni-kl.de}

\setkomafont{pageheadfoot}{%
\normalfont\normalcolor\itshape\small
}
\setkomafont{pagenumber}{\normalfont\bfseries}

\def\stackunder#1#2{\mathrel{\mathop{#2}\limits_{#1}}}

\makeatletter\@addtoreset{equation}{section}\makeatother

\allowdisplaybreaks

\theoremstyle{plain}      \newtheorem{theorem}{Theorem}[section]
                          \newtheorem{corollary}[theorem]{Corollary}
                          \newtheorem{proposition}[theorem]{Proposition}
													\newtheorem{condition}[theorem]{Condition}

\theoremstyle{remark}     \newtheorem{remark}[theorem]{Remark}
                          \newtheorem{lemma}[theorem]{Lemma}

\theoremstyle{definition} \newtheorem{definition}[theorem]{Definition}

\begin{document} 

\newcommand{\grad}{\nabla}
\newcommand{\D}{\partial}
\newcommand{\E}{\mathcal{E}}
\newcommand{\N}{\mathbb{N}}
\newcommand{\R}{\mathbb{R}_{\scriptscriptstyle{\ge 0}}}
\newcommand{\dom}{\mathcal{D}}
\newcommand{\ess}{\operatorname{ess~inf}}
\newcommand{\cem}{\operatorname{\text{\ding{61}}}}
\newcommand{\supp}{\operatorname{\text{supp}}}
\newcommand{\ca}{\operatorname{\text{cap}}}

\begin{titlepage}
\title{\Large Construction of a finite volume dynamical wetting model with $\delta$-pinning in (d+1)-dimension via Dirichlet forms}
\author{\normalsize\sc Torben Fattler\footnote{University of Kaiserslautern, P.O.Box 3049, 67653
Kaiserslautern, Germany.}~\thanks{\emailfattler}~\footnotemark[4] \and \normalsize\sc Martin Grothaus\footnotemark[1]~\thanks{\emailgrothaus}~\thanks{\urluni} 
 \and \normalsize\sc Robert Vo\ss hall\thanks{\emailvosshall}}
\date{\small\today}
\end{titlepage}
\maketitle

\pagestyle{headings}

\begin{abstract}
We give a Dirichlet form approach for the construction of a distorted Brownian motion in $E:=[0,\infty)^n$, $n\in\mathbb{N}$, where the behavior on the boundary is determined by the competing effects of reflection from and pinning at the boundary. The problem is formulated in an $L^2$-setting with underlying measure $\mu=\varrho\,m$. Here $\varrho$ is a positive density, integrable with respect to the measure $m$ and fulfilling the Hamza condition. The measure $m$ is such that the boundary $\partial E$ of $E$ is not of $m$-measure zero. A reference measure $\mu$ of this type is needed in order to give meaning to the so-called Wentzell boundary condition which is in literature typical for modeling such kind of boundary behavior. In providing a Skorokhod decomposition of the constructed process we are able to justify that the stochastic process is solving the underlying stochastic differential equation weakly in the sense of N.~Ikeda and Sh.~Watanabe for $\mathcal{E}$-quasi every starting point. At the boundary the constructed process indeed is governed by the competing effects of reflection and pinning. In order to obtain the Skorokhod decomposition we need $\varrho$ to be continuously differentiable on $E$, which is equivalent to continuity of the logarithmic derivative of $\varrho$. Furthermore, we assume that the logarithmic derivative of $\varrho$ is square integrable with respect to $\mu$. We do not need that the logarithmic derivative of $\varrho$ is Lipschitz continuous.
In particular, our considerations enable us to construct a dynamical wetting model (also known as Ginzburg--Landau dynamics) on a bounded set $D_{\scriptscriptstyle{N}}\subset \mathbb{Z}^d$ under mild assumptions on the underlying pair interaction potentials in all dimensions $d\in\mathbb{N}$. In dimension $d=2$ this model describes the motion of an interface resulting from wetting of a solid surface by a fluid.
\\\\
\\
\end{abstract}

\section{Introduction}
In numerous physical situations it happens that different states of aggregation of a certain matter appear simultaneously. In particular, such phenomena occur at low temperature. Thereby, the transitions to different states of aggregation are separated by fairly sharp interfaces. The so-called $\nabla\phi$ interface model provides a fundamental mathematical model for the physical description of such interfaces from a microscopic or mesoscopic point of view. We are interested in the time development of such interfaces. In \cite{FuSpo97} the authors consider a scalar field $\phi^t$, $t\ge 0$, where its motion is governed by a reversible stochastic dynamics. I.e., in a finite volume $\Lambda\subset \mathbb{Z}^d$, $d\in\mathbb{N}$, under suitable boundary conditions, the scalar field $\phi^t:=\big(\phi_t(x)\big)_{x\in\Lambda}$, $t\ge 0$, is described by the stochastic differential equations
\begin{align*}
d\phi_t(x)=-\sum_{\stackunder{\scriptscriptstyle{|x-y|_{\scriptscriptstyle{\text{euc}}}=1}}{\scriptscriptstyle{y\in\Lambda}}}V'(\phi_t(x)-\phi_t(y))dt+\sqrt{2}\,dB_t(x),\quad x\in\Lambda,\quad t\ge 0.
\end{align*}
Here $|\cdot|_{\scriptscriptstyle{\text{euc}}}$ denotes the norm induced by the euclidean scalar product on $\mathbb{R}^d$, $V\in C^2(\mathbb{R})$ is a symmetric, strictly convex potential and $\big\{(B_t(x))_{t\ge 0}\,|\,x\in\Lambda\big\}$ are independent Brownian motions. Such a dynamics is known as the \emph{Ginzburg-Landau $\nabla\phi$ interface model in finite volume}.
Of particular interest in the framework of $\nabla\phi$ interface models is the so-called \emph{entropic repulsion}. Though one considers  the $\nabla\phi$ interface model with reflection on a hard wall. This phenomenon was investigated e.g.~in \cite{DeuGia00} and \cite{BDG01} for the static $\nabla\phi$ interface model. Interface motion with entropic repulsion, i.e., the Ginzburg-Landau $\nabla\phi$ interface model with entropic repulsion was studied recently in \cite{DeuNis07} for dimension $d\ge 2$. Here the underlying potentials are again symmetric, strictly convex and nearest neighbor $C^2$-pair potentials. The Ginzburg-Landau dynamics with repulsion was introduced by T.~Funaki and S.~Olla in \cite{Fu03, FuOl01}. In \cite{Za04} this problem was tackled via Dirichlet form techniques in dimension $d=1$.

In considering the $\nabla\phi$ interface model with reflection on a hard wall and additionally putting a pinning effect on that wall, we are dealing with the so-called \emph{wetting model}. In dimension $d=2$ this model describes the wetting of a solid surface by a fluid. The static wetting model was studied recently in \cite{DGZ05}, see also \cite{CaVe00}. Considerations of the Ginzburg-Landau dynamics with reflection on a hard wall under the influence of an outer force, causing e.g.~a mild pinning effect on the wall can be found in \cite{Fu03}. 

In \cite[Sect.~15.1]{Fu05} J.-D. Deuschel and T. Funaki investigated the scalar field $\phi^t:=\big(\phi_t(x)\big)_{x\in\Lambda}$, $t\ge 0$, described by the stochastic differential equations
\begin{multline}\label{sde}
d\phi_t(x)=-\mathbbm{1}_{(0,\infty)}\big(\phi_t(x)\big)\sum_{\stackunder{\scriptscriptstyle{|x-y|_{\scriptscriptstyle{\text{euc}}}=1}}{\scriptscriptstyle{y\in\Lambda}}}V'\big(\phi_t(x)-\phi_t(y)\big)\,dt\\
+\mathbbm{1}_{(0,\infty)}\big(\phi_t(x)\big)\sqrt{2}dB_t(x)+d\ell_{t}^{\scriptscriptstyle{0}}(x),\quad x\in\Lambda,
\end{multline}
subject to the conditions:
\begin{align*}
&\phi_t(x)\ge 0,\quad \ell_{t}^{\scriptscriptstyle{0}}(x)\mbox{ is non-decreasing with respect to }t,\quad \ell^{\scriptscriptstyle{0}}_{0}(x)=0,\\
&\int_0^\infty\phi_t(x)\,d\ell_{t}^{\scriptscriptstyle{0}}(x)=0,\\
&s\ell_{t}^{\scriptscriptstyle{0}}(x)=\int_0^t\mathbbm{1}_{\{0\}}\big(\phi_{\tau}(x)\big)\,d\tau\quad\mbox{for fixed }s\ge 0,\nonumber\\
\end{align*}
where $\ell_{t}^{\scriptscriptstyle{0}}(x)$ denotes the \emph{local time} of $\phi_t(x)\mbox{ at }0$ and the pair interaction potential $V$ is again symmetric, strictly convex and $C^2$.

For treating this system of stochastic differential equations the authors gave reference to classical solution techniques as developed e.g.~in \cite{WaIk89}. The methods provided therein require more restrictive assumptions on the drift part as in our situation (we only need that the drift part is continuous, not necessarily Lipschitz continuous), moreover, do not apply directly (the geometry differs). First steps in the direction of applying \cite{WaIk89} are discussed in \cite{Fu05} by J.-D. Deuschel and T. Funaki. First we use Dirichlet form techniques in order to construct solutions in the sense of the associated martingale problem for general Wentzell type boundary conditions. Then, by providing a Skorokhod decomposition for the constructed process, we can show that this process solves the stochastic differential equation 
\begin{multline}\label{repintro}
d\mathbf{X}_{t}^j=\mathbbm{1}_{\mathring{E}}\big(\mathbf{X}_t\big)\,\sqrt{2}\,dB^j_t+\partial_j\ln(\varrho)\big(\mathbf{X}_\tau\big)\mathbbm{1}_{\mathring{E}}\big(\mathbf{X}_\tau\big)\,dt\\
+\sum_{\varnothing\not=B\subsetneq I}\left\{\begin{array}{ll}
  \mathbbm{1}_{{E_+(B)}}\big(\mathbf{X}_t\big)\,\sqrt{2}\,dB^j_t+\partial_j\ln(\varrho)\big(\mathbf{X}_\tau\big)\mathbbm{1}_{{E_+(B)}}\big(\mathbf{X}_\tau\big)\,dt, & \text{if }j\in B\\
  \frac{1}{s}\,\mathbbm{1}_{{E_+(B)}}\big(\mathbf{X}_\tau\big)\,dt, & \text{if }j\in I\setminus B
  \end{array}\right.\\
+\frac{1}{s}\,\mathbbm{1}_{\{(0,\ldots,0)\}}\big(\mathbf{X}_\tau\big)\,dt
\end{multline}
weakly in the sense of N.~Ikeda and Sh.~Watanabe, see e.g.~\cite{WaIk89}, for $\mathcal{E}$-quasi every starting point. Here $E:=[0,\infty)^n$, $n\in\mathbb{N}$, is the state space, $j\in I:=\{1,\ldots,n\}$, $B\subsetneq I$, $E_+(B):=\big\{x\in E\,|\,x_i>0\text{ for all }i\in B\big\}\subset \partial E$, $(B^j_t)_{t\ge 0}$ are one dimensional Brownian motions with $B^j_0=0$, $j\in I$, and $\varrho$ is a positive, continuously differentiable density on $E$ such that $\sqrt{\varrho}$ is in the Sobolev space of weakly differentiable functions on $\mathring{E}$, square integrable together with their derivative. $\varrho$ continuously differentiable on $E$ is equivalent to the drift part $\big(\partial_j\ln(\varrho)\big)_{j\in I}$ being continuous. The stochastic differential equation (\ref{repintro}) can be rewritten as  
\begin{align*}
d\mathbf{X}_{t}^j=\mathbbm{1}_{(0,\infty)}\big(\mathbf{X}^j_t\big)\,\Big(\sqrt{2}\,dB^j_t+\partial_j\ln(\varrho)\big(\mathbf{X}_t\big)\,dt\Big)+\frac{1}{s}\,\mathbbm{1}_{\{0\}}\big(\mathbf{X}^j_t\big)\,dt,\quad j\in I,
\end{align*}
or equivalently
\begin{multline}\label{repintro2}
d\mathbf{X}_{t}^j=\mathbbm{1}_{(0,\infty)}\big(\mathbf{X}^j_t\big)\,\Big(\sqrt{2}\,dB^j_t+\partial_j\ln(\varrho)\big(\mathbf{X}_t\big)\,dt\Big)+d\ell_0^j\\
\text{with}\quad \ell_0^j:=\frac{1}{s}\int_0^t\mathbbm{1}_{\{0\}}\big(\mathbf{X}_{\tau}^j\big)\,d\tau,\quad j\in I.
\end{multline}
Hence we obtain a weak solution to (\ref{sde}) in the sense of N.~Ikeda and Sh.~Watanabe under rather mild assumption on the underlying pair interaction potentials.
 
A Skorokhod decomposition for reflected diffusions on bounded Lipschitz domains with singular non-reflection part was provided by G. Trutnau in \cite{Tru03}. Here we consider the case of Wentzell type boundary condition. 
As far as we know our considerations are the first with regard to construction and analysis of the dynamical wetting model in finite volume. Our approach is based on Dirichlet form techniques. Since we obtain a one to one correspondence to the functional analytic tool, as Dirichlet form and operator semigroup, we expect this method in future to be useful in view of studying scaling limits for the underlying system of stochastic differential equations as done in e.g.~\cite{Za04}. Dirichlet form methods in the context of Wentzell boundary condition were introduced in e.g.~\cite{VoVo03}. Here, however, in view of our application we construct via the underlying bilinear form a dynamics even on the boundary. In \cite{VoVo03} a static boundary behavior is realized. An overview of the state of the art in the framework of interface models is presented in e.g.~\cite{Ga02}, \cite{Fu05}.

Our paper is organized as follows. In Section \ref{sectfuana} we provide the functional analytical background to apply Dirichlet form methods in order to tackle the problem of the dynamical wetting model. We analyze the bilinear form (\ref{form1}) and prove in Theorem \ref{theosumdiri} that $\big(\mathcal{E},D(\mathcal{E})\big)$ is a conservative, strongly local, regular, symmetric Dirichlet form on $L^2\big(E;\mu_{n,s,\varrho}\big)$. In Section \ref{sectprocess} we present the probabilistic counterpart of Section \ref{sectfuana}. The main result of this section is obtained in Theorem \ref{theoprocess}, where we show that $\big(\mathcal{E},D(\mathcal{E})\big)$ has an associated conservative diffusion process $\mathbf{M}$, i.e., an associated strong Markov process with continuous sample paths and infinite life time. Since $\mathbf{M}$ solves the martingale problem, see Theorem \ref{theomartingale}, for the corresponding generator $\big(H,D(H)\big)$ it can be considered as the solution of a stochastic differential equation. The diffusion process $\mathbf{M}$ is analyzed in Section \ref{sectanapro}. Here we provide in Corollary \ref{coroskoro} a Skorokhod decomposition of $\mathbf{M}$. Finally, we apply our results to the problem of the dynamical wetting model, see Theorem \ref{theosumdiriapp}, Theorem \ref{theoprocessapp}, Theorem \ref{theomartingaleapp} and Corollary \ref{coroskoroapp}.

The following list of main results summarizes the progress achieved in this paper.
\begin{enumerate}
\item[(i)]
We construct conservative diffusion processes in $[0,\infty)^n$, $n\in\mathbb{N}$, with the competing effects of reflection and pinning at the boundary under mild assumptions on the drift part, see Theorem \ref{theoprocess} and Theorem \ref{theomartingale}.
\item[(ii)]
We provide a Skorokhod decomposition of the constructed process and thereby justify that the process solves the underlying stochastic differential equations weakly in the sense of N.~Ikeda and Sh.~Watanabe for $\mathcal{E}$-quasi every starting point. Moreover, we illustrate the behavior of the process at the boundary, see Corollary \ref{coroskoro}.
\item[(iii)]
Our general considerations apply to the construction of the dynamical wetting model in finite volume and all dimensions $d\in\mathbb{N}$ for a large class of pair interaction potentials, see Theorem \ref{theoprocessapp}, Theorem \ref{theomartingaleapp} and Corollary \ref{coroskoroapp}.
\end{enumerate}

\section{The functional analytical background}\label{sectfuana}
Let $n\in\mathbb{N}$, $I:=I_n:=\big\{1,\ldots,n\big\}$ and $E:=E_n:=[0,\infty)^n$. We have that $\mathring{E}=(0,\infty)^n$ and we denote by $\partial E$ the boundary of $E$. 
For each $x\in{E}$ we set
\begin{align*}
I_{\scriptscriptstyle{0}}(x):=\big\{i\in I\,\big|\,x_i=0\big\}\quad\mbox{and}\quad I_{\scriptscriptstyle{+}}(x):=\big\{i\in I\,\big|\,x_i>0\big\},
\end{align*}
and define for $A,B\subset I$,
\begin{align*}
{E}_{\scriptscriptstyle{0}}(A):=\Big\{x\in E\,\Big|\,I_{\scriptscriptstyle{0}}(x)=A\Big\}\quad\mbox{and}\quad {E}_{\scriptscriptstyle{+}}(B):=\Big\{x\in E\,\Big|\,I_{\scriptscriptstyle{+}}(x)=B\Big\},
\end{align*}
respectively.
\begin{remark}
We have the decomposition
\begin{align*}
{E}=\dot{\bigcup}_{A\subset I}{E}_{\scriptscriptstyle{0}}(A)=\dot{\bigcup}_{B\subset I}{E}_{\scriptscriptstyle{+}}(B). 
\end{align*}
In particular,
\begin{align*}
\partial E={E}\setminus \mathring{E}=\dot{\bigcup}_{\varnothing\not=A\subset I}{E}_{\scriptscriptstyle{0}}(A)=\dot{\bigcup}_{B\subsetneq I}{E}_{\scriptscriptstyle{+}}(B).
\end{align*}
\end{remark}
On $\big({E},\mathcal{B}_{\scriptscriptstyle{{E}}}\big)$ with $\mathcal{B}_{\scriptscriptstyle{{E}}}$ being the trace $\sigma$-algebra of the Borel $\sigma$-algebra $\mathcal{B}(\mathbb{R}^n)$ on ${E}$ we define for fixed $s\in(0,\infty)$ the measures 
\begin{align}\label{defmeasure}
m_{n,s}:=\sum_{B\subset I}\lambda_{\scriptscriptstyle{B}}^{n,s}\quad\text{with}\quad\lambda_{\scriptscriptstyle{B}}^{n,s}:=s^{n-\# B}\,\lambda_{\scriptscriptstyle{B}}^{\scriptscriptstyle{(n)}}\quad\text{and}\quad\lambda_{\scriptscriptstyle{B}}^{\scriptscriptstyle{(n)}}:=\prod_{i\in B}dx_{\scriptscriptstyle{+}}^i\prod_{j\in I\setminus B}d\delta^j_{\scriptscriptstyle{0}},
\end{align}
where $\#S$ gives the number of elements in a set $S$, $dx^i_{\scriptscriptstyle{+}}$ is the Lebesgue measure on\\ $\big([0,\infty),\mathcal{B}\big([0,\infty)\big)\big)$ and $\delta_{\scriptscriptstyle{0}}^j$ denotes the Dirac measure on $\big([0,\infty),\mathcal{B}\big([0,\infty)\big)\big)$ at $0$. The indices $i,j\in I$ give reference to the component of $x=(x_1,\ldots,x_n)\in E$ being integrated by $dx^i_{\scriptscriptstyle{+}}$ and $\delta^j_{\scriptscriptstyle{0}}$, respectively. 

\begin{condition}\label{conddensity}
$\varrho$ is a $m_{n,s}$-a.e. positive function on ${E}$ such that $\varrho\in L^{1}\big({E};m_{n,s}\big)$.
\end{condition}

\begin{remark}\label{remprobdensity}
In particular, $\varrho$ can be chosen to be a probability density.
\end{remark}
Under Condition \ref{conddensity} we define on $\big(E,\mathcal{B}_{\scriptscriptstyle{E}}\big)$ the measure $\mu_{n,s,\varrho}:=\varrho\,m_{n,s}$ and hence, the space of square integrable functions on $E$ with respect to $\mu_{n,s,\varrho}$, denoted by $L^2\big(E;\mu_{n,s,\varrho}\big)$. 

\begin{remark}\label{remBaire}
Note that the measure $\mu_{n,s,\varrho}$ on $\big(E,\mathcal{B}_{\scriptscriptstyle{E}}\big)$ is a \emph{Baire measures}. In our setting this means $\mu_{n,s,\varrho}$ is a Borel measure with the additional property that
\begin{align}\label{Baire}
\mu_{n,s,\varrho}(K)<\infty\quad\text{for all compact sets}\quad K\subset E.
\end{align}
(\ref{Baire}) is fulfilled, since $\varrho\in L^{1}\big(E;m_{n,s,\varrho}\big)$. Obviously, $E$ is locally compact and countable at infinity.
\end{remark}

We set
\begin{align*}
C^{{0}}_{{c}}\big(E\big):=\Big\{f:E\to\mathbb{R}\,\Big|\,f\mbox{ is continuous on }{E}
\mbox{ with }\supp(f)\subset E\mbox{ compact}\Big\},
\end{align*}
where $\supp$ denotes the support of the corresponding function and for $k\in\mathbb{N}$ we define
\begin{multline*}
C^{{k}}_{{c}}\big(E\big):=\Big\{f:E\to\mathbb{R}\,\Big|\,f\mbox{ is $k$-times continuously differentiable on }\mathring{E}\\
\mbox{ with }\supp(f)\subset E\mbox{ compact}\mbox{ and }\\
\partial^lf\mbox{ extends continuously to }E\mbox{ for $|l|\le k$}\Big\}.
\end{multline*}
Here and below $\partial^lf$ denotes the partial derivative of $f$ to the multi index $l\in\mathbb{N}^{{n}}_{\scriptscriptstyle{0}}$, i.e., 
\begin{multline*}
l=\big(l_1,\ldots,l_n\big)\in\mathbb{N}^{{n}}_{\scriptscriptstyle{0}},\quad |l|=l_1+\ldots+l_n,\\
\partial^lf:=\partial_1^{l_1}\partial_2^{l_2}\ldots\partial_n^{l_n}f,\quad\partial_i^{l_i}f:=\partial_{x_i}^{l_i}f,\quad \partial_{x_i}^0f:=f,\quad i\in I.
\end{multline*}
In case $k=1$ we write $\partial_i$ instead of $\partial^1_i$. Furthermore, $C_c^\infty({E}):=\bigcap_{k\in\mathbb{N}_{\scriptscriptstyle{0}}}C^k_c(E)$.

\begin{proposition}\label{propdensedef}
Under Condition \ref{conddensity} we have that $C_c^\infty\big(E\big)$ is dense in $L^2\big(E;\mu_{n,s,\varrho}\big)$.
\end{proposition}

\begin{proof}
Let $f\in L^2\big(E;\mu_{n,s,\varrho}\big)$. Due to Remark \ref{remBaire} we can apply \cite[Coro.~7.5.5]{Bau81} to obtain that $C_c^0\big(E\big)$ is dense in $L^2\big(E;\mu_{n,s,\varrho}\big)$. Hence there exists a sequence $(g_i)_{i\in\mathbb{N}}$ in $C_c^0\big(E\big)$ such that $\lim_{i\to\infty}\Vert f-g_i\Vert_{L^2(E;\mu_{n,s,\varrho})}=0$. The extended Stone--Weierstra\ss~theorem, see e.g.~\cite[Chap.~7, Sect.~38]{Sim63}, yields that $C_c^\infty\big(E\big)$ is dense in $C_c^0\big(E\big)$ with respect to $\Vert\cdot\Vert_{\sup}$, where $\Vert f\Vert_{\sup}:=\sup_{x\in E}|f(x)|$ for $f\in C_c^0\big(E\big)$. Thus to each $i\in\mathbb{N}$ there exists a sequences $\big(f^i_j\big)_{j\in\mathbb{N}}$ in $C_c^\infty\big(E\big)$ such that $\lim_{j\to\infty}\Vert g_i-f^i_j\Vert_{\sup}=0$. Without loss of generality we can assume that there exists a compact set $K_i\subset E$ such that $\text{supp}(g_i),\text{supp}({f}^{i}_j)\subset K_{i}$ for all $j\in\mathbb{N}$. Let $\varepsilon >0$. Then there exists $g_k\in C^0_c\big(E\big)$ such that $\Vert f-g_k\Vert_{L^2(E;\mu_{n,s,\varrho})}<\frac{1}{2}\sqrt{\varepsilon}$. Furthermore, there exists $f^{k}_{l}\in C_c^\infty\big(E\big)$ such that $\Vert g_k-f_{l}^k\Vert_{\sup}<\frac{1}{2}\sqrt{\frac{\varepsilon}{\mu_{n,s,\varrho}(K_k)}}$, where $\mu_{n,s,\varrho}(K_k)\in(0,\infty)$ by Condition \ref{conddensity}. Hence
\begin{multline*}
\big\Vert f-f^k_l\big\Vert_{L^2(E;\mu_{n,s,\varrho})}^2\le2\,\big\Vert f-g_k\big\Vert_{L^2(E;\mu_{n,s,\varrho})}^2+2\,\big\Vert g_k-f^k_l\big\Vert_{L^2(E;\mu_{n,s,\varrho})}^2\\
<\frac{\varepsilon}{2}+2\int_{E}\big(g_k-f^k_l\big)^2\,d\mu_{n,s,\varrho}\le \frac{\varepsilon}{2}+2\,\Vert g_k-f^k_l\Vert_{\sup}^2\cdot \mu_{n,s,\varrho}(K_k)<\frac{\varepsilon}{2}+\frac{\varepsilon}{2}=\varepsilon.
\end{multline*} 
Therefore, $C_c^\infty\big(E\big)$ is dense in $L^2\big(E;\mu_{n,s,\varrho}\big)$.
\end{proof}

\subsection{Dirichlet forms}
For fixed $n\in\mathbb{N}$, $s\in(0,\infty)$ and $\varrho$ fulfilling Condition \ref{conddensity} we define on $L^2\big(E;\mu_{n,s,\varrho}\big)$ the bilinear form
\begin{align}\label{form1}
\mathcal{E}(f,g):=\mathcal{E}^{n,s,\varrho}\big(f,g\big):=\sum_{\varnothing\not=B\subset I}\mathcal{E}_{\scriptscriptstyle{B}}(f,g),\quad f,g\in \mathcal{D}:=C_c^2\big(E\big),
\end{align}
with 
\begin{align*}
\mathcal{E}_{\scriptscriptstyle{B}}(f,g):=\mathcal{E}_{\scriptscriptstyle{B}}^{n,s,\varrho}\big(f,g\big):=\sum_{i\in B}\int_{E_{\scriptscriptstyle{+}}(B)}\partial_if\,\partial_i g\,d\mu_{\scriptscriptstyle{B}}^{\varrho,n,s},\quad\varnothing\not=B\subset I,
\end{align*}
where $\mu_{\scriptscriptstyle{B}}^{\varrho,n,s}:=\varrho\,\lambda_{\scriptscriptstyle{B}}^{n,s}$ (see (\ref{defmeasure})).
\begin{proposition}\label{propdense}
Suppose that Condition \ref{conddensity} is satisfied. Then $\big(\mathcal{E},\mathcal{D}\big)$ is a symmetric, positive definite bilinear form which is densely defined on $L^2\big(E;\mu_{n,s,\varrho}\big)$.
\end{proposition}
\begin{proof}
That the bilinear form is densely defined we obtain by Proposition \ref{propdensedef}. The rest follows directly by definition.
\end{proof}

To prove closability of the underlying bilinear form, we have to put an additional restriction on the density $\varrho$. For $\varnothing\not= B\subset I$ we define
\begin{align*}
R_\varrho\big(\overline{E_{\scriptscriptstyle{+}}(B)}\big):=\left\{x\in \overline{E_{\scriptscriptstyle{+}}(B)}\,\left|\,\int_{B_{\varepsilon}(x)}\varrho^{-1}\,d\lambda_{\scriptscriptstyle{B}}^{n,s}<\infty\quad\text{for some}\quad\varepsilon>0\right.\right\},
\end{align*}
where $B_{\varepsilon}(x):=\big\{y\in \overline{E_{\scriptscriptstyle{+}}(B)}\,\big|\,|x-y|_{\scriptscriptstyle{\text{euc}}}\le\varepsilon\big\}$ and $|\cdot|_{\scriptscriptstyle{\text{euc}}}$ denotes the norm induced by the euclidean scalar product on $\mathbb{R}^n$ and for $\varnothing\not=B\subset I$, $\overline{E_{\scriptscriptstyle{+}}(B)}$ is the closure of ${E_{\scriptscriptstyle{+}}(B)}$ with respect to $|\cdot|_{\scriptscriptstyle{\text{euc}}}$. 

\begin{condition}\label{condHamza}
For $\varnothing\not=B\subset I$ we have that $\varrho=0$~$\lambda_{\scriptscriptstyle{B}}^{n,s}$-a.e.~on $\overline{E_{\scriptscriptstyle{+}}(B)}\setminus R_\varrho\big(\overline{E_{\scriptscriptstyle{+}}(B)}\big)$.
\end{condition}

\begin{lemma}\label{lemMaRo}
Let Condition \ref{condHamza} be satisfied. For $\varnothing\not=B\subset I$ let $\varphi\in C^\infty_c\Big(R_\varrho\big(\overline{E_{\scriptscriptstyle{+}}(B)}\big)\Big)$ and $f\in L^2\big(E,\mu_{n,s,\varrho}\big)$.
\begin{enumerate}
\item[(i)]
There exists $C_1(\varphi,B)\in(0,\infty)$ such that
\begin{align*}
\Bigg|\int_{R_\varrho(\overline{E_{\scriptscriptstyle{+}}(B)})}f\varphi\,d\lambda_{\scriptscriptstyle{B}}^{n,s}\Bigg|\le C_1(\varphi,B)\cdot\Vert f\Vert_{L^2(\overline{E_{\scriptscriptstyle{+}}(B)};\mu_{\scriptscriptstyle{B}}^{\varrho,n,s})}.
\end{align*}
Here $L^2\big(\overline{E_{\scriptscriptstyle{+}}(B)};\mu_{\scriptscriptstyle{B}}^{\varrho,n,s}\big)$ denotes the spaces of square integrable functions on $\overline{E_{\scriptscriptstyle{+}}(B)}$ with respect to $\mu_{\scriptscriptstyle{B}}^{\varrho,n,s}$.
\item[(ii)]
There exists $C_2(\varphi,B)\in(0,\infty)$ such that
\begin{align*}
\Bigg|\int_{\partial R_\varrho(\overline{E_{\scriptscriptstyle{+}}(B)})}f\varphi\,d\sigma_{\scriptscriptstyle{B}}^{n,s}\Bigg|\le C_2(\varphi,B)\cdot\Vert f\Vert_{L^2(E;\mu_{\varrho,n,s})},
\end{align*}
where $\sigma_{\scriptscriptstyle{B}}^{n,s}:=\sum_{\hat{B}\subsetneq B}\lambda_{\scriptscriptstyle{\hat{B}}}^{n,s}$.
\end{enumerate}
\end{lemma}

\begin{proof}
\begin{enumerate}
\item[(i)]
See e.g.~\cite[Chap.~2,~Lemm. 2.2]{MR92}.
\item[(ii)]
Let $\varnothing\not=B\subset I$, $\varphi\in C^\infty_c\Big(R_\varrho\big(\overline{E_{\scriptscriptstyle{+}}(B)}\big)\Big)$ and $f\in L^2\big(\overline{\Omega},\mu_{n,s,\varrho}\big)$. By a multiple application of part (i) we obtain
\begin{multline*}
\Bigg|\int_{\partial R_\varrho(\overline{E_{\scriptscriptstyle{+}}(B)})}f\varphi\,d\sigma_{\scriptscriptstyle{B}}^{n,s}\le
\sum_{\hat{B}\subsetneq B}\Bigg|\int_{\partial R_\varrho(\overline{E_{\scriptscriptstyle{+}}(B)})}f\varphi\,d\lambda_{\scriptscriptstyle{\hat{B}}}^{n,s}\Bigg|\\
\le\sum_{\hat{B}\subsetneq B}C_1(\varphi,\hat{B})\cdot\Vert f\Vert_{L^2(\overline{E_{\scriptscriptstyle{+}}(\hat{B})};\mu_{\scriptscriptstyle{\hat{B}}}^{\varrho,n,s})}
\le C_2(\varphi,B)\cdot\Vert f\Vert_{L^2(E;\mu_{\varrho,n,s})}.
\end{multline*}
\end{enumerate}
\end{proof}

\begin{proposition}\label{propclos}
Suppose that Conditions \ref{conddensity} and \ref{condHamza} are satisfied. Then $\big(\mathcal{E},\mathcal{D}\big)$ is closable on $L^2\big(E;\mu_{n,s,\varrho}\big)$. Its closure we denote by $\big(\mathcal{E},D(\mathcal{E})\big)$.
\end{proposition}

\begin{proof}
Let $(f_k)_{k\in\mathbb{N}}$ be a Cauchy sequence in $\mathcal{D}$ with respect to $\mathcal{E}$, i.e.,~$\mathcal{E}(f_k-f_l,f_k-f_l)\to 0$ as $k,l\to\infty$. Furthermore, we suppose that $f_k\to 0$ in $L^2\big(E;\mu_{n,s,\varrho}\big)$ as $k\to\infty$, i.e.,
$\big(f_k,f_k\big)_{L^2(E;\mu_{n,s,\varrho})}\to 0$ as $k\to\infty$. We have to check whether $\mathcal{E}(f_k,f_k)\to 0$ as $k\to\infty$. 
Let $\varnothing\not=B\subset I$. We know that for fixed $i\in B$, $\big(\partial_i f_k\big)_{k\in\mathbb{N}}$ converges to some $h_i$ in $L^2\big(\overline{E_{\scriptscriptstyle{+}}(B)};\mu_{\scriptscriptstyle{B}}^{\varrho,n,s}\big)$, since $\big(\partial_i f_k\big)_{k\in\mathbb{N}}$ is a Cauchy sequence in $L^2\big(\overline{E_{\scriptscriptstyle{+}}(B)};\mu_{\scriptscriptstyle{B}}^{\varrho,n,s}\big)$ and $\big(L^2\big(\overline{E_{\scriptscriptstyle{+}}(B)};\mu_{\scriptscriptstyle{B}}^{\varrho,n,s}\big),\Vert\cdot\Vert_{L^2(\overline{E_{\scriptscriptstyle{+}}(B)};\mu_{\scriptscriptstyle{B}}^{\varrho,n,s})}\big)$ is complete. Let $\varphi\in C_c^\infty\big(R_\varrho\big(\overline{E_{\scriptscriptstyle{+}}(B)}\big)\big)$. Using Lemma \ref{lemMaRo}(i) we obtain that    
\begin{multline*}
\Bigg|\int_{R_\varrho(\overline{E_{\scriptscriptstyle{+}}(B)})}h_i\,\varphi\,d\lambda_{\scriptscriptstyle{B}}^{n,s}-\int_{R_\varrho(\overline{E_{\scriptscriptstyle{+}}(B)})}\partial_if_k\,\varphi\,d\lambda_{\scriptscriptstyle{B}}^{n,s}\Bigg|\\
=\Bigg|\int_{R_\varrho(\overline{E_{\scriptscriptstyle{+}}(B)})}\Big(h_i-\partial_i f_k\Big)\,\varphi\,d\lambda_{\scriptscriptstyle{B}}^{n,s}\Bigg|\le C_1(\varphi,B)\cdot\Vert h_i-\partial_i f_k\Vert_{L^2(\overline{E_{\scriptscriptstyle{+}}(B)};\mu_{\scriptscriptstyle{B}}^{\varrho,n,s})}\to 0\quad\text{as}\quad k\to\infty.
\end{multline*}
This, together with an integration by parts, triangle inequality, Lemma \ref{lemMaRo}(ii) and the fact that $\big(f_k,f_k\big)_{L^2(E;\mu_{n,s,\varrho})}\to 0$ as $k\to\infty$ implies:
\begin{multline*}
\left|\int_{R_\varrho(\overline{E_+(B)})}h_i\,\varphi\,d\lambda_{\scriptscriptstyle{B}}^{n,s}\right|=\lim_{k\to\infty} \left|\int_{R_\varrho(\overline{E_+(B)})}\partial_if_k\,\varphi\,d\lambda_{\scriptscriptstyle{B}}^{n,s}\right|\\
=\lim_{k\to\infty} \left|\int_{\partial(R_\varrho(\overline{E_+(B)}))}f_k\,\varphi\,d\sigma_{\scriptscriptstyle{B}}^{n,s}-\int_{R_\varrho(\overline{E_+(B)})} f_k\,\partial_i\varphi\,d\lambda_{\scriptscriptstyle{B}}^{n,s}\right|\\
\le\lim_{k\to\infty} \left|\int_{\partial(R_\varrho(\overline{E_+(B)}))}f_k\,\varphi\,d\sigma_{\scriptscriptstyle{B}}^{n,s}\right|+\lim_{k\to\infty}\left|\int_{R_\varrho(\overline{E_+(B)})} f_k\,\partial_i\varphi\,d\lambda_{\scriptscriptstyle{B}}^{n,s}\right|
=0\quad\text{as}\quad k\to\infty.
\end{multline*}
Thus $h_i=0$ in $L^2\big(R_\varrho(\overline{E_{\scriptscriptstyle{+}}(B)});\lambda_{\scriptscriptstyle{B}}^{n,s}\big)$ and therefore $h_i=0$ in $L^2\big(\overline{E_{\scriptscriptstyle{+}}(B)};\varrho\,\lambda_{\scriptscriptstyle{B}}^{n,s}\big)$ by Condition \ref{condHamza}. For all $\varnothing\not=B\subset I$ this yields $h_i=0$ in $L^2\big(\overline{E_{\scriptscriptstyle{+}}(B)};\mu_{\scriptscriptstyle{B}}^{\varrho,n,s}\big)$ for all $i\in B$. Moreover,
\begin{multline*}
\mathcal{E}(f_k,f_k)=\sum_{\varnothing\not=B\subset I}\int_{E_{\scriptscriptstyle{+}}(B)}\sum_{i\in B}\big(\partial_if_k\big)^2\,d\mu_{\scriptscriptstyle{B}}^{\varrho,n,s}\\
=\sum_{\varnothing\not=B\subset I}\sum_{i\in B}\big\Vert\partial_i f_k-h_i\big\Vert_{L^2(\overline{E_{\scriptscriptstyle{+}}(B)};\mu_{\scriptscriptstyle{B}}^{\varrho,n,s})}^2\to 0\quad\text{as}\quad k\to\infty
\end{multline*}
and closability is shown. 
\end{proof}

\begin{remark}
Since $\big(\mathcal{E},\mathcal{D}\big)$ is closable on $L^2\big(E;\mu_{n,s,\varrho}\big)$ by Proposition \ref{propclos} we have that $D(\mathcal{E})$ is complete with respect to the norm $\Vert\cdot\Vert_{\mathcal{E}_1}:={\mathcal{E}(\cdot,\cdot)}^{\frac{1}{2}}+{\big(\cdot,\cdot\big)^{\frac{1}{2}}_{L^2(E;\mu_{n,s,\varrho})}}$.
\end{remark}

\begin{proposition}\label{propDirichlet}
Suppose that Conditions \ref{conddensity} and \ref{condHamza} are satisfied. Then $\big(\mathcal{E},D(\mathcal{E})\big)$ is a symmetric, regular Dirichlet form.
\end{proposition}

\begin{proof}
The Markov property is clear, see e.g. \cite[Chap.~2, Sect.~2,~Example c)]{MR92}. Regularity can be shown as follows.
The extended Stone--Weierstra\ss~theorem, see e.g.~\cite[Chap.~7, Sect.~38]{Sim63}, yields that $C_c^\infty\big(E\big)$ is dense in $C_c^0\big(E\big)$ with respect to $\Vert\cdot\Vert_{\sup}$. Furthermore, $\mathcal{D}$ is dense in $D\big(\mathcal{E}\big)$ with respect to $\Vert\cdot\Vert_{\mathcal{E}_1}$. Since $C_c^\infty(E)\subset\mathcal{D}\subset D(\mathcal{E})\cap C^0_c(E)$, we obtain that $\big(\mathcal{E},D(\mathcal{E})\big)$ is regular.
\end{proof}

\begin{proposition}\label{propstronglocal}
Suppose that Conditions \ref{conddensity} and \ref{condHamza} are satisfied. Then the regular, symmetric Dirichlet form $\big(\mathcal{E},D(\mathcal{E})\big)$ is strongly local and conservative.
\end{proposition}

\begin{proof}
Using \cite[Theo.~3.1.1]{FOT94} and \cite[Problem~3.1.1]{FOT94} it is sufficient to show the strong local property for elements in $\mathcal{D}$. Therefore, let $f,g\in\mathcal{D}$ with $\supp(f)$, $\supp(g)$ compact and let $g$ be constant on some open (in the trace topology of $E$) neighborhood $U$ of $\supp(f)$. Then
\begin{multline*}
\mathcal{E}\big(f,g\big)=\sum_{\varnothing\not=B\subset I}\sum_{i\in B}\int_{E_{\scriptscriptstyle{+}}(B)}\partial_if\,\partial_i g\,d\mu_{\scriptscriptstyle{B}}^{\varrho,n,s}\\
=\sum_{\varnothing\not=B\subset I}\sum_{i\in B}\int_{E_{\scriptscriptstyle{+}}(B)\cap\supp(f)}\partial_if\,\underbrace{\partial_i g}_{=0}\,d\mu_{\scriptscriptstyle{B}}^{\varrho,n,s}+\sum_{\varnothing\not=B\subset I}\sum_{i\in B}\int_{E_{\scriptscriptstyle{+}}(B)\setminus\supp(f)}\underbrace{\partial_if}_{=0}\,\partial_i g\,d\mu_{\scriptscriptstyle{B}}^{\varrho,n,s}=0.
\end{multline*} 
Hence $\big(\mathcal{E},D(\mathcal{E})\big)$ is strongly local.

Next we prove conservativity. $\mathbbm{1}_{E}\in L^2\big(E;\mu_{n,s,\varrho}\big)$ by Condition \ref{conddensity}. We show that $\mathbbm{1}_{E}\in D(\mathcal{E})$. Set $\Lambda:=[-1,\infty)^n$ and $K_k:=[0,k]^n$, $k\in\mathbb{N}$. Then there exist cutoff functions $f_k\in C_c^\infty(\Lambda)$, $k\in\mathbb{N}$, such that $0\le f_k\le f_{k+1}\le 1$, $f_k=1$ on $K_k$, $\supp(f_k)\subset B_1(K_k)$ and $\big|\partial_i f_k\big|\le C_3<\infty$. $C_3$ independent of $k\in\mathbb{N}$. Here $B_1(K_k)$, $k\in\mathbb{N}$, denotes the $1$-neighborhood of the set $K_k$. Hence
\begin{multline}\label{conv1}
\Vert\mathbbm{1}_{E}-f_k\Vert_{L^2(E;\mu_{n,s,\varrho})}^2=\int_{E}\big(\mathbbm{1}_{E}-f_k\big)^2\,d\mu_{n,s,\varrho}\\
=\int_{E\setminus K_k}\big(\mathbbm{1}_{E}-f_k\big)^2\,d\mu_{n,s,\varrho}\le \mu_{n,s,\varrho}\big(E\setminus K_k\big)\to 0\quad\text{as}\quad k\to\infty.
\end{multline}
Furthermore,
\begin{multline}\label{conv2}
\mathcal{E}(f_k,f_k)=\sum_{\varnothing\not=B\subset I}\int_{E_{\scriptscriptstyle{+}}(B)}\sum_{i\in B}\big(\partial_if_k\big)^2\,d\mu_{\scriptscriptstyle{B}}^{\varrho,n,s}\le n\,\int_{E\setminus K_k}\big(\partial_i f_k\big)^2\,d\mu_{n,s,\varrho}\\
\le n\,C_3^2\,\mu_{n,s,\varrho}\big(E\setminus K_k\big)\to 0\quad\text{as}\quad k\to\infty.
\end{multline} 
Using (\ref{conv1}) and (\ref{conv2}) we easily obtain by applying Cauchy-Schwarz inequality that $(f_k)_{k\in\mathbb{N}}$ is $\mathcal{E}_1$-Cauchy. Hence $\mathbbm{1}_{E}\in D(\mathcal{E})$ with
\begin{align*}
\mathcal{E}(\mathbbm{1}_E,\mathbbm{1}_E)=\lim_{k\to\infty}\mathcal{E}(f_k,f_k)=0
\end{align*}
 and conservativity is shown.
\end{proof}

Finally, we end up with the following result.

\begin{theorem}\label{theosumdiri}
For fixed $n\in\mathbb{N}$, $s\in(0,\infty)$ and density function $\varrho$ we have that under Conditions \ref{conddensity} and \ref{condHamza} 
\begin{align*}
\mathcal{E}(f,g)=\sum_{\varnothing\not=B\subset I}\mathcal{E}_{\scriptscriptstyle{B}}(f,g),\quad f,g\in \mathcal{D}=C_c^2\big(E\big),
\end{align*}
with 
\begin{align*}
\mathcal{E}_{\scriptscriptstyle{B}}(f,g)=\sum_{i\in B}\int_{E_{\scriptscriptstyle{+}}(B)}\partial_if\,\partial_i g\,d\mu_{\scriptscriptstyle{B}}^{\varrho,n,s},\quad\varnothing\not=B\subset I,
\end{align*}
and $\mu_{\scriptscriptstyle{B}}^{\varrho,n,s}=\varrho\,\lambda_{\scriptscriptstyle{B}}^{n,s}$,
is a densely defined, positive definite, symmetric bilinear form, which is closable on $L^2\big(E;\mu_{n,s,\varrho}\big)$. Its closure $\big(\mathcal{E},D(\mathcal{E})\big)$ is a conservative, strongly local, regular, symmetric Dirichlet form on $L^2\big(E;\mu_{n,s,\varrho}\big)$.
\end{theorem}
\begin{proof}
See Propositions \ref{propdense}, \ref{propclos}, \ref{propDirichlet} and \ref{propstronglocal}.
\end{proof}

\subsection{Generators}
By Friedrichs representation theorem we have the existence of the self-adjoint generator\\
$\big(H,D(H)\big)$ corresponding to $\big(\mathcal{E},D(\mathcal{E})\big)$.
\begin{proposition}\label{propgen}
Suppose that Conditions \ref{conddensity} and \ref{condHamza} are satisfied. There exists a unique, positive, self-adjoint, linear operator $\big(H,D(H)\big)$ on $L^2\big(E;\mu_{n,s,\varrho}\big)$ such that
\begin{align*}
D(H)\subset D(\mathcal{E})\quad\text{and}\quad\mathcal{E}\big(f,g\big)=\Big(H f,g\Big)_{L^2(E;\mu_{n,s,\varrho})}
\quad\text{for all }f\in D(H),~g\in D(\mathcal{E}).
\end{align*}
\end{proposition}
\begin{proof}
Using Proposition \ref{propclos} this is a direct application of \cite[Coro.~1.3.1]{FOT94}.
\end{proof}

We need additional assumptions on the density function $\varrho$ in order to give a characterization of the generator $H$ on a certain subset of its domain $D(H)$. 

\begin{condition}\label{condweakdiff}
$ $
\begin{enumerate}
\item[(i)]
$\sqrt{\varrho}\in H^{1,2}(\mathring{E})$, where $H^{1,2}(\mathring{E})$ denotes the Sobolev space of weakly differentiable functions on $\mathring{E}$, square integrable together with their derivative.
\item[(ii)]
$\varrho\in C^1(E)$, where $C^1(E)$ denotes the space of continuously differentiable functions on $E$.
\end{enumerate}
\end{condition}

\begin{remark}\label{remcondequi}
$ $
\begin{enumerate}
\item[(i)]
Note that the additional assumptions collected in Condition \ref{condweakdiff} are not necessary for the existence of the generator $\big(H,D(H)\big)$.
\item[(ii)]
If $\varrho$ fulfills Condition \ref{conddensity} then Condition \ref{condweakdiff}(i) is equivaltent to $\big(\partial_i\ln(\varrho)\big)_{i=1}^n$\\
$\in L^2(E;\mu^{n,s,\varrho})$.
\item[(iii)]
Condition \ref{condweakdiff}(ii) is equivalent to $\big(\partial_i\ln(\varrho)\big)_{i=1}^n$ being continuous.
\end{enumerate}
\end{remark}

For $f\in\mathcal{D}$ we define
\begin{align*}
L^sf:={L}^{n,s,\varrho}f:=\sum_{i=1}^n\partial^2_i f+\sum_{i=1}^n\partial_i f\,\partial_i (\ln\varrho)+\frac{1}{s}\sum_{i=1}^n\partial_i f
\end{align*}
and
\begin{align*}
Lf:={L}^{n,{\scriptscriptstyle{\infty}},\varrho} f:=\sum_{i=1}^n\partial^2_i f+\sum_{i=1}^n\partial_i f\,\partial_i (\ln\varrho).
\end{align*}
\begin{proposition}\label{propibp}
Suppose Conditions \ref{conddensity}, \ref{condHamza} and \ref{condweakdiff} are satisfied. For functions 
\begin{align*}
f,g\in \mathcal{D}_{\scriptscriptstyle{\text{Wentzell}}}:=\Big\{h\in\mathcal{D}\,\Big|\,\big(L^{s}h\big)\big|_{{E_{\scriptscriptstyle{+}}(B)}}=0\quad\text{for all}\quad B\subsetneq I\Big\}
\end{align*}
we have the representation $\mathcal{E}\big(f,g\big)=\Big(-Lf,g\Big)_{\scriptscriptstyle{L^2(E;\mu_{n,s,\varrho})}}$.
\end{proposition}

\begin{remark}
Elements from $\mathcal{D}_{\scriptscriptstyle{\text{Wentzell}}}$ are said to fulfill a \emph{Wentzell type boundary condition}.
\end{remark}

\begin{proof}
Let $f,g\in C^2_c(E)$. In order to show this representation we carry out an integration by parts. We start with $B=I$, i.e., $\# B=n$:
\begin{multline*}
\mathcal{E}_{\scriptscriptstyle{I}}(f,g)=\sum_{i\in I}\int_{\mathring{E}}\partial_if\,\partial_i g\,\varrho\,d\lambda^{\scriptscriptstyle{(n)}}_{\scriptscriptstyle{I}}
=\sum_{i\in I}\int_{\mathring{E}}\partial_if\varrho\,\partial_i g\,d\lambda_{\scriptscriptstyle{I}}^{\scriptscriptstyle{(n)}}\\
=\sum_{i\in I}\int_{\mathring{E}}\Big(-\partial^2_if\varrho-\partial_if\partial_i\varrho\Big)\,
g(x)\,d\lambda^{\scriptscriptstyle{(n)}}_{\scriptscriptstyle{I}}-\sum_{\stackunder{\#B=n-1}{B\subset I}}\sum_{i\in I\setminus B}\int_{E_+(B)}\partial_i fg\varrho\,d\lambda^{\scriptscriptstyle{(n)}}_{\scriptscriptstyle{B}}\\
=\sum_{i\in I}\int_{\scriptscriptstyle{\mathring{E}}}\Big(-\partial^2_i f-\partial_i f\partial_i \ln(\varrho)\Big)\,
g(x)\,\varrho\,d\lambda_{\scriptscriptstyle{I}}^{\scriptscriptstyle{(n)}}-\sum_{\stackunder{\#B=n-1}{B\subset I}}\sum_{i\in I\setminus B}\int_{E_+(B)}\partial_i f\,g\varrho\,d\lambda_{\scriptscriptstyle{B}}^{\scriptscriptstyle{(n)}}.
\end{multline*}
Next we consider all $B\subset I$ such that $\# B=n-1$, i.e.,
\begin{multline*}
\mathcal{E}_{\scriptscriptstyle{B}}(f,g)=\sum_{i\in B}\int_{E_{+}(B)}\partial_if\,\partial_i g\,s\,\varrho\,\prod_{i\in B}dx_{+}^i\prod_{j\in I\setminus B}d\delta^j_0\\
=\sum_{i\in B}\int_{E_+(B)}\Big(-\partial^2_i f-\partial_i f\partial_i \ln(\varrho)\Big)\,
g\,\varrho\,s\,d\lambda_{\scriptscriptstyle{B}}^{\scriptscriptstyle{(n)}}\\
-\sum_{\stackunder{\#\tilde{B}=n-2}{\tilde{B}\subset B}}\sum_{i\in I\setminus \tilde{B}}\int_{E_+(\tilde{B})}\partial_i f\,g\varrho\,s\,d\lambda_{\scriptscriptstyle{\tilde{B}}}^{\scriptscriptstyle{(n)}}.
\end{multline*}
Proceeding inductively we end up with all $B\subset I$ fulfilling $\# B=1$, i.e., we consider
\begin{multline*}
\mathcal{E}_{\scriptscriptstyle{B}}(f,g)=\sum_{i\in B}\int_{E_{+}(B)}\partial_if\,\partial_i g\,\varrho\,s^{n-1}\,\prod_{i\in B}dx_{+}^i\prod_{j\in I\setminus B}d\delta^j_0\\
=\sum_{i\in B}\int_{E_+(B)}\Big(-\partial^2_i f-\partial_i f\partial_i \ln(\varrho)\Big)\,
g\,\varrho\,s^{n-1}\,d\lambda_{\scriptscriptstyle{B}}^{\scriptscriptstyle{(n)}}\\
-s^{n-1}\,\sum_{i\in I}\partial_i f(0)\,g(0)\varrho(0).
\end{multline*}
Combining all this yields
\begin{multline}\label{ibp}
\mathcal{E}\big(f,g\big)=\sum_{\varnothing\not=B\subset I}\mathcal{E}_{\scriptscriptstyle{B}}\big(f,g\big)\\
=\sum_{i\in I}\int_{\mathring{E}}\Big(-\partial^2_i f-\partial_i f\partial_i \ln(\varrho)\Big)\,
g\,\varrho\,d\lambda_{\scriptscriptstyle{I}}^{\scriptscriptstyle{(n)}}\\
+\sum_{\stackunder{\#B=n-1}{B\subset I}}\int_{E_+(B)}\left(\sum_{i\in B}\Big(-\partial^2_i f-\partial_i f\partial_i \ln(\varrho)\Big)-\frac{1}{s}\sum_{i\in I\setminus B}\partial_i f\right)\,
\,g\,s\,\varrho\,d\lambda_{\scriptscriptstyle{B}}^{\scriptscriptstyle{(n)}}\\
+\sum_{\stackunder{\#B=n-2}{B\subset I}}\int_{E_+(B)}\left(\sum_{i\in B}\Big(-\partial^2_i f-\partial_i f\partial_i \ln(\varrho)\Big)-\frac{1}{s}\sum_{i\in I\setminus B}\partial_i f\right)\,
\,g\,s^2\,\varrho\,d\lambda_{\scriptscriptstyle{B}}^{\scriptscriptstyle{(n)}}\\
\begin{array}{c}
+\\
\vdots\\
+
\end{array}\\
\sum_{\stackunder{\#B=1}{B\subset I}}\int_{E_+(B)}\left(\sum_{i\in B}\Big(-\partial^2_i f-\partial_i f\partial_i \ln(\varrho)\Big)-\frac{1}{s}\sum_{i\in I\setminus B}\partial_i f\right)\,
\,g\,s^{n-1}\,\varrho(x)\,d\lambda_{\scriptscriptstyle{B}}^{\scriptscriptstyle{(n)}}\\
-\sum_{i\in I}\frac{1}{s}\,\partial_i f(0)\,g(0)\,s^n\,\varrho(0).
\end{multline}
Now using Wentzell type boundary condition we obtain the desired result.
\end{proof}

\section{The associated Markov process}\label{sectprocess}
Since $\big(\mathcal{E},D(\mathcal{E})\big)$ is a regular, symmetric Dirichlet form on $L^2\big(E;\mu_{n,s,\varrho}\big)$ which is conservative and possesses the strong local property, we obtain the following theorem, where $\big(T_t\big)_{t>0}$ denotes the $C^0$-semigroup corresponding to $\big(\mathcal{E},D(\mathcal{E})\big)$, see e.g.~\cite[Diagram 3]{MR92}.
\begin{theorem}\label{theoprocess}
Suppose that Conditions \ref{conddensity} and \ref{condHamza} are satisfied. Then there exists a conservative diffusion process (i.e.~a strong Markov process with continuous sample paths and infinite life time)
\begin{align*}
\mathbf{M}:=\mathbf{M}^{n,s,\varrho}:=\big(\mathbf{\Omega},\mathbf{F},(\mathbf{F}_t)_{t\ge 0},(\mathbf{X}_t)_{t\ge 0},(\mathbf{\Theta}_t)_{t\ge 0},(\mathbf{P}^{n,s,\varrho}_x)_{x\in E}\big)
\end{align*}
with state space $E$ which is properly associated with $\big(\mathcal{E},D(\mathcal{E})\big)$, i.e., for all ($\mu_{n,s,\varrho}$-versions of) $f\in L^2\big(E;\mu_{n,s,\varrho}\big)$ and all $t>0$ the function
\begin{align*}
E\ni x\mapsto\mathbb{E}^{n,s,\varrho}_x\Big(f\big(\mathbf{X}_t\big)\Big):=\int_{\mathbf{\Omega}}f\big(\mathbf{X}_t\big)\,d\mathbf{P}^{n,s,\varrho}_x\in[0,\infty)
\end{align*}
is a $\mathcal{E}$-quasi continuous version of $T_tf$. $\mathbf{M}$ is up to $\mu_{n,s,\varrho}$-equivalence unique. In particular, $\mathbf{M}$ is $\mu_{n,s,\varrho}$-symmetric, i.e., 
\begin{align*}
\int_{E}T_tf\,g\,d\mu_{n,s,\varrho}=\int_{E}f\,T_tg\,d\mu_{n,s,\varrho}\quad\text{for all}\quad f,g:E\to[0,\infty)\text{ measurable and all }t>0,
\end{align*}
and has $\mu_{n,s,\varrho}$ as invariant measure, i.e., 
\begin{align*}
\int_{E}T_tf\,d\mu_{n,s,\varrho}=\int_{E}f\,d\mu_{n,s,\varrho}\quad\text{for all}\quad f:E\to[0,\infty)\text{ measurable and all }t>0.
\end{align*}
\end{theorem}
In the above theorem $\mathbf{M}$ is canonical, i.e., $\mathbf{\Omega}=C^0\big([0,\infty),E\big)$, the space of continuous functions on $[0,\infty)$ into $E$, $\mathbf{X}_t(\omega)=\omega(t)$, $\omega\in\mathbf{\Omega}$. The filtration $(\mathbf{F}_t)_{t\ge 0}$ is the natural minimum completed admissible filtration obtained from the $\sigma$-algebras $\mathbf{F}^0_t:=\sigma\Big\{\textbf{X}_{\tau}\,\Big|\,0\le \tau\le t\Big\}$, $t\ge 0$, and $\mathbf{F}:=\mathbf{F}_{\scriptscriptstyle{\infty}}:=\bigvee_{t\in[0,\infty)}\mathbf{F}_t$. For each $t\ge 0$ we denote by $\mathbf{\Theta}_t:\mathbf{\Omega}\to\mathbf{\Omega}$ a shift operator such that $\mathbf{X}_\tau\circ\mathbf{\Theta}_t=\mathbf{X}_{\tau+t}$ for all $\tau\ge 0$.
\begin{proof}
See e.g.~\cite[Chap.~V,~Theo.~1.11]{MR92}.
\end{proof}

\begin{theorem}\label{theomartingale}
The diffusion process $\mathbf{M}$ from Theorem \ref{theoprocess} is up to $\mu_{n,s,\varrho}$-equivalence the unique diffusion process having $\mu_{n,s,\varrho}$ as symmetrizing measure and solving the martingale problem for $\big(H,D(H)\big)$, i.e., for all $g\in D(H)$
\begin{align*}
\widetilde{g}(\mathbf{X}_t)-\widetilde{g}(\mathbf{X}_0)+\int_0^t \big(Hg\big)(\mathbf{X}_\tau)\,d\tau,\quad t\ge 0,
\end{align*}
is an $\mathbf{F}_t$-martingale under $\mathbf{P}^{n,s,\varrho}_x$ (hence starting in x) for $\mathcal{E}$-quasi all $x\in E$. $\mathcal{E}$-quasi all $x\in E$ or $\mathcal{E}$-quasi every $x\in E$ (abbreviated by q.a.~$x\in E$ or q.e.~$x\in E$, respectively) means all $x\in E$ except those contained in a set of $\mathcal{E}$-capacity zero. 
(Here $\widetilde{g}$ denotes a quasi-continuous version of $g$, see \cite[Chap.~IV, Prop.~3.3]{MR92}.)
\end{theorem}
\begin{proof}
See e.g.~\cite[Theo.~3.4(i)]{AR95}.
\end{proof}

\section{Analysis of the stochastic process by additive functionals}\label{sectanapro}  
Throughout this section we assume that we are given the regular, symmetric Dirichlet form $\big(\mathcal{E},D(\mathcal{E})\big)$ on $L^2\big(E;\mu_{n,s,\varrho}\big)$ which is conservative and possesses the strong local property, see Section \ref{sectfuana}, and the associated diffusion process $\mathbf{M}$ from Section \ref{sectprocess}. 
Let $g\in D(\mathcal{E})$ be essentially bounded. Due to \cite[Sect.~3.2]{FOT94} there exists a unique, finite, positive Radon measure $\nu_{\scriptscriptstyle{\langle g\rangle}}$ on $\big(E,\mathcal{B}_{E}\big)$ satisfying
\begin{align*}
\int_{E}f\,d\nu_{\scriptscriptstyle{\langle g\rangle}}=2\,\mathcal{E}(gf,g)-\mathcal{E}(g^2,f)\quad\text{for all }f\in D(\mathcal{E})\cap C^{0}_c\big(E\big).
\end{align*}
\begin{remark}
For an essentially bounded $g\in D(\mathcal{E})$ the measure $\nu_{\scriptscriptstyle{\langle g\rangle}}$ is called the \emph{energy measure} of $g$.
\end{remark}

\begin{proposition}\label{propenergymeasure}
Suppose that Conditions \ref{conddensity} and \ref{condHamza} are satisfied. Let $g\in\mathcal{D}\subset D(\mathcal{E})$. Then the energy measure $\nu_{\scriptscriptstyle{\langle g\rangle}}$ of $g$ is given by
\begin{align*}
\nu_{\scriptscriptstyle{\langle g\rangle}}=2\sum_{\varnothing\not=B\subset I}\sum_{i\in B}\big(\partial_ig\big)^2\,\mu_{\scriptscriptstyle{B}}^{\varrho,n,s}.
\end{align*}
\end{proposition}

\begin{proof}
First let $f,g\in\mathcal{D}$. Then we have
\begin{multline*}
2\,\mathcal{E}(gf,g)-\mathcal{E}(g^2,f)\\
=2\sum_{\varnothing\not=B\subset E}\,\int_{E_+(B)}\sum_{i\in B}\partial_i\big( gf\big)\,\partial_i g\,d\mu_{\scriptscriptstyle{B}}^{\varrho,n,s}-\sum_{\varnothing\not=B\subset I}\int_{E_+(B)}\sum_{i\in B}\partial_i\big(g^2\big)\,\partial_i f\,d\mu_{\scriptscriptstyle{B}}^{\varrho,n,s}\\
=2\sum_{\varnothing\not=B\subset I}\int_{E_+(B)}\sum_{i\in B}\Big(\partial_i gf+g\partial_i f\Big)\,\partial_i g\,d\mu_{\scriptscriptstyle{B}}^{\varrho,n,s}-2\sum_{\varnothing\not=B\subset I}\int_{E_+(B)}\sum_{i\in B}g\,\partial_i g\,\partial_i f\,d\mu_{\scriptscriptstyle{B}}^{\varrho,n,s}\\
=2\sum_{\varnothing\not=B\subset I}\int_{E_+(B)}\sum_{i\in B}\Big(\big(\partial_i g\big)^2\,f+g\,\partial_i g\,\partial_i f\Big)\,d\mu_{\scriptscriptstyle{B}}^{\varrho,n,s}-2\sum_{\varnothing\not=B\subset I}\int_{E_+(B)}\sum_{i\in B}g\,\partial_i g\,\partial_i f\,d\mu_{\scriptscriptstyle{B}}^{\varrho,n,s}\\
=2\sum_{\varnothing\not=B\subset I}\int_{E_+(B)}\sum_{i\in B}\big(\partial_i g\big)^2\,f\,d\mu_{\scriptscriptstyle{B}}^{\varrho,n,s}=\int_{E}f\,\,2\sum_{\varnothing\not=B\subset I}\sum_{i\in B}\big(\partial_i g\big)^2\,d\mu_{\scriptscriptstyle{B}}^{\varrho,n,s}.
\end{multline*}
Now let $g\in\mathcal{D}$ and $f\in D(\mathcal{E})\cap C_c^0\big(E\big)$. Moreover, let $(f_k)_{k\in\mathbb{N}}\subset\mathcal{D}$ such that $f_k\to f$ in $\big(D(\mathcal{E}),\Vert\cdot\Vert_{\mathcal{E}_1}\big)$ as $k\to\infty$. One easily verifies that $f_kg\to fg$ in $\big(D(\mathcal{E}),\Vert\cdot\Vert_{\mathcal{E}_1}\big)$ as $k\to\infty$. Hence
\begin{multline*}
2\,\mathcal{E}\big(fg,g\big)-\mathcal{E}\big(g^2,f\big)=\lim_{k\to\infty}2\,\mathcal{E}\big(f_k g,g\big)-\lim_{k\to\infty}\mathcal{E}\big(g^2,f_k\big)\\
=\lim_{k\to\infty}\Big(2\,\mathcal{E}\big(f_k g,g\big)-\mathcal{E}\big(g^2,f_k\big)\Big)=2\lim_{k\to\infty}\int_{E}f_k\sum_{\varnothing\not=B\subset I}\sum_{i\in B}\big(\partial_i g\big)^2\,d\mu_{\scriptscriptstyle{B}}^{\varrho,n,s}\\
=2\,\sum_{\varnothing\not=B\subset I}\lim_{k\to\infty}\int_{E_+(B)}f_k\,\sum_{i\in B}\big(\partial_i g\big)^2\,d\mu_{\scriptscriptstyle{B}}^{\varrho,n,s}=2\,\sum_{\varnothing\not=B\subset I}\int_{E_+(B)}f\,\sum_{i\in B}\big(\partial_i g\big)^2\,d\mu_{\scriptscriptstyle{B}}^{\varrho,n,s},
\end{multline*}
since for $\varnothing\not=B\subset I$,
\begin{multline*}
\lim_{k\to\infty}\left|\sum_{i\in B}\int_{E_+(B)}f_k\,\big(\partial_i g\big)^2\,d\mu_{\scriptscriptstyle{B}}^{\varrho,n,s}-\sum_{i\in B}\int_{E_+(B)}f\,\big(\partial_i g\big)^2\,d\mu_{\scriptscriptstyle{B}}^{\varrho,n,s}\right|\\
\le\lim_{k\to\infty}\sum_{i\in B}\int_{E}\left|f_k-f\right|\,\big(\partial_i g\big)^2\,d\mu_{\scriptscriptstyle{B}}^{\varrho,n,s}\\
\le\lim_{k\to\infty}\big\Vert f_k-f\big\Vert_{L^2(E_+(B);\mu_{\scriptscriptstyle{B}}^{\varrho,n,s})}\,\underbrace{\sum_{i\in B}\big\Vert \big(\partial_ig\big)^2\big\Vert_{L^2(E_+(B);\mu_{\scriptscriptstyle{B}}^{\varrho,n,s})}}_{<\infty}\to 0\quad\text{as }k\to\infty,
\end{multline*}
by H\"older inequality. Hence 
\begin{align*}
2\,\mathcal{E}(gf,g)-\mathcal{E}(g^2,f)=\int_{E}f\,\,2\sum_{\varnothing\not=B\subset I}\sum_{i\in B}\big(\partial_i g\big)^2\,d\mu_{\scriptscriptstyle{B}}^{\varrho,n,s},\quad g\in\mathcal{D}\text{ and }f\in D(\mathcal{E})\cap C_c^0\big(E\big).
\end{align*}
\end{proof}

\begin{proposition}\label{propmeasureform}
Suppose Conditions \ref{conddensity}, \ref{condHamza} and \ref{condweakdiff} are satisfied. Let $f,g\in\mathcal{D}\subset D(\mathcal{E})$. Then
\begin{align*}
\mathcal{E}\big(f,g\big)=\big\langle\nu_f,g\big\rangle:=\int_{E}g\,d\nu_f
\end{align*}
with 
\begin{align*}
\nu_f:=\sum_{B\subset I}\Bigg(\sum_{i\in B}\Big(-\partial^2_i f-\partial_i f\partial_i \ln(\varrho)\Big)\Bigg)\,\lambda^{n,s,\varrho}_{\scriptscriptstyle{B}}-\frac{1}{s}\sum_{B\subsetneq I}\Big(\sum_{i\in I\setminus B}\partial_i f\Big)\,\lambda^{n,s,\varrho}_{\scriptscriptstyle{B}}.
\end{align*}
\end{proposition}

\begin{proof}
This representation is valid due to the integration by parts carried out in the proof of Proposition \ref{propibp}.
\end{proof}

Next we recall the definition of a \emph{positive, continuous, additive functional} (see e.g.~\cite[Appendix A.2,~A.3]{FOT94}).

\begin{definition}[\emph{additive functional}]
A family $\big(A_t\big)_{t\ge 0}$ of extended real valued functions $A_t:\Omega\to\overline{\mathbb{R}}$, with $\overline{\mathbb{R}}:=\mathbb{R}\cup\{-\infty,\infty\}$, is called \emph{additive functional} (\emph{A}F in abbreviation) of $\mathbf{M}$ if it satisfies the following conditions:
\begin{enumerate}
\item[(A1)]
$A_t$ is $\mathbf{F}_t$-measurable for each $t\ge 0$.
\item[(A2)]
There exists $\Lambda\in\mathbf{F}_{\scriptscriptstyle{\infty}}$ with $\mathbf{P}^{n,s,\varrho}_x(\Lambda)=1$ for all $x\in E$, $\mathbf{\Theta}_t\Lambda\subset\Lambda$ for all $t>0$ and for each $\omega\in\Lambda$, $t\mapsto A_t(\omega)$ is right continuous and has left limit on $\big[0,\infty\big)$ satisfying
\begin{enumerate}
\item[(i)]
$A_0(\omega)=0$, and 
\item[(ii)]
$A_{t+\tau}(\omega)=A_t(\omega)+A_\tau(\mathbf{\Theta}_t\omega)$ for all $t,\tau\ge 0$.
\end{enumerate}
\end{enumerate}
The set $\Lambda$ in the above is called a \emph{defining set} for $\big(A_t\big)_{t\ge 0}$. An $\big(A_t\big)_{t\ge 0}$ is said to be \emph{finite} if $\big|A_t(\omega)\big|<\infty$ for all $t\in[0,\infty)$ and each $\omega$ in a defining set.
An $\big(A_t\big)_{t\ge 0}$ is said to be \emph{continuous} if $[0,\infty)\ni t\mapsto A_t(\omega)\in\overline{\mathbb{R}}$ is continuous for each $\omega$ in a defining set. A continuous AF $\big(A_t\big)_{t\ge 0}$ consisting of a family of $[0,\infty]$-valued functions $A_t:\Omega\to[0,\infty]$ is called a \emph{positive continuous AF} (\emph{PCAF} in abbreviation). The set of all PCAFs we denote by $\mathbf{A}^+_{c}$. Moreover, we call an AF which is also a square integrable martingale with respect to $\big(\mathbf{F}_t\big)_{t\ge 0}$ a \emph{martingale AF} (\emph{MAF} in abbreviation).  
\end{definition}

\begin{lemma}\label{lemPACF}
Suppose that Conditions \ref{conddensity} and \ref{condHamza} are satisfied. Let $0\le g\in C^0\big(E\big)$ and $M\in\mathcal{B}_{\scriptscriptstyle{E}}$ Then $A:=(A_t)_{t\ge 0}$ with
\begin{align*}
A_t(\omega):=\int_0 ^{t}g\big(\mathbf{X}_{\tau}(\omega)\big)\,\mathbbm{1}_{M}\big(\mathbf{X}_{\tau}(\omega)\big)\,d\tau,\quad\omega\in \Omega,
\end{align*}
is a PCAF, i.e., $A\in\mathbf{A}^+_{c}$. If $g$ is bounded, $A$ is even finite.
\end{lemma}

\begin{proof}
Clearly, $A_t$ is $\mathbf{F}_t$-measurable for each $t\ge 0$. Since $\mathbf{M}$ is a diffusion process of infinite life time we can choose as defining set $\Lambda$ the set of all $\omega\in\Omega$ such that $\mathbf{X}_{\cdot}(\omega)$ is continuous. For all $t_1,t_2\ge 0$ and $\omega\in\Omega$ we have
\begin{multline*}
A_{t_1+t_2}(\omega)=\int_0^{t_1+t_2}g\big(\mathbf{X}_{\tau}\big)\mathbbm{1}_M\big(\mathbf{X}_{\tau}(\omega)\big)\,d\tau\\
=\int_0^{t_2}g\big(\mathbf{X}_{\tau}\big)\mathbbm{1}_M\big(\mathbf{X}_{\tau}(\omega)\big)\,d\tau+\int_{t_2}^{t_1+t_2}g\big(\mathbf{X}_{\tau}\big)\mathbbm{1}_M\big(\mathbf{X}_{\tau}(\omega)\big)\,d\tau\\
=A_{t_2}(\omega)+\int_{t_2}^{t_1+t_2}g\big(\mathbf{X}_{\tau-t_2}\circ\mathbf{\Theta}_{t_2}(\omega)\big)\mathbbm{1}_M\big(\mathbf{X}_{\tau-t_2}\circ\mathbf{\Theta}_{t_2}(\omega)\big)\,d\tau\\
=A_{t_2}(\omega)+\int_{0}^{t_1}g\big(\mathbf{X}_{\tau}\circ\mathbf{\Theta}_{t_2}(\omega)\big)\mathbbm{1}_M\big(\mathbf{X}_{\tau}\circ\mathbf{\Theta}_{t_2}(\omega)\big)\,d\tau
=A_{t_2}(\omega)+A_{t_1}(\mathbf{\Theta}_{t_2}\big(\omega)\big)
\end{multline*}
and in particular, $A_0(\omega)=0$. $A$ is a continuous AF, since $g$ is continuous. $A$ is clearly positive and in case of bounded $g$ finite.
\end{proof}
Given $\mathbf{M}$ and a positive measure $\mu$ on $\big(E,\mathcal{B}_{\scriptscriptstyle{E}}\big)$ we define a positive measure $\mathbf{P}_\mu$ on $(\mathbf{\Omega},\mathbf{F})$ by
\begin{align*}
\mathbf{P}_\mu(\Gamma):=\int_E\mathbf{P}^{n,s,\varrho}_x(\Gamma)\,d\mu(x),\quad\Gamma\in\mathbf{F}.
\end{align*}
Now we want to assign to the measures $\nu_{\scriptscriptstyle{\langle g\rangle}}$ from Proposition \ref{propenergymeasure} and $\nu_f$ from Proposition \ref{propmeasureform} the corresponding \emph{additive functionals (AFs)}. In order to do this we make use of \cite[Theo.~5.1.3]{FOT94}. 

We consider the following classes of measures.

 \begin{definition}[\emph{smooth measure}, \emph{measure of finite energy integral}]
We denote by $\mathbf{S}$ the family of \emph{smooth measures}, i.e., all positive Borel measures $\mu$ on $\mathcal{B}_{\scriptscriptstyle{E}}$ such that $\mu$ charges no set of $\mathcal{E}$-capacity zero and there exists an increasing sequence $\big(F_k\big)_{k\in\mathbb{N}}$ of closed sets in $E$ such that $\mu\big(F_k\big)<\infty$ for all $k\in\mathbb{N}$ and $\lim_{k\to\infty}\text{cap}_{\mathcal{E}}\big(K\setminus F_k\big)=0$ for any compact set $K\subset E$. Here $\text{cap}_{\mathcal{E}}\big(S\big)$ denotes the $\mathcal{E}$-capacity of a set $S\subset E$.  

A positive Radon measure $\mu$ on $\mathcal{B}_{\scriptscriptstyle{E}}$ is said to be of \emph{finite energy integral} if
\begin{align*}
\int_E\big|f\big|\,d\mu\le C_4\sqrt{\mathcal{E}_1\big(f,f\big)},\quad f\in D(\mathcal{E})\cap C_c(E),
\end{align*}
for some $C_4\in(0,\infty)$. We denote by $\mathbf{S}_{\scriptscriptstyle{0}}$ the set of all positive Radon measures of finite energy integrals.
\end{definition}

\begin{remark}\label{remUpot}
A positive Radon measure $\mu$ on $\mathcal{B}_{\scriptscriptstyle{E}}$ is of finite energy integral if and only if there exists for each $\alpha>0$ a unique $U_\alpha\mu\in D\big(\mathcal{E}\big)$ such that
\begin{align*}
\mathcal{E}_\alpha\big(U_\alpha\mu,f\big)=\int_E f\,d\mu\quad\mbox{for all }f\in D\big(\mathcal{E}\big)\cap C_c(E),
\end{align*}
where $\mathcal{E}_\alpha(\cdot,\cdot):={\mathcal{E}(\cdot,\cdot)}+{\alpha\,\big(\cdot,\cdot\big)_{L^2(E;\mu_{n,s,\varrho})}}$.
\end{remark}

\begin{definition}[\emph{$\alpha$-potential}]
We call $U_\alpha\mu$ from Remark \ref{remUpot} an \emph{$\alpha$-potential} and denote by $\mathbf{S}_{\scriptscriptstyle{00}}$ the set of all finite $\mu\in \mathbf{S}_{\scriptscriptstyle{0}}$ such that $\Vert U_1\mu\Vert_{\scriptscriptstyle{L^\infty(E;\mu_{n,s,\varrho})}}<\infty$. 
\end{definition}

\begin{lemma}\label{lemS00}
Let $\mu\in\mathbf{S}_{\scriptscriptstyle{00}}$ be a finite measure and $g:E\to [0,\infty)$ measurable and bounded. Then ${\mu_g}:=g\mu\in\mathbf{S}_{\scriptscriptstyle{00}}$.
\end{lemma}

\begin{proof}
Since the finite measure $\mu$ is of finite energy integral and $g$ is bounded, ${\mu_g}$ is bounded and of finite energy integral. Let $h\in D\big(\mathcal{E}\big)\cap C_c(E)$ such that $h\ge 0$. We have that $g\le C_5$ for some $C_5\in(0,\infty)$ and
\begin{align*}
\mathcal{E}_1\big(U_1{\mu_g},h\big)=\int_E h\,d{\mu_g}=\int_Eh\,g\,d\mu\le C_5\int_E h\,d\mu=C_5\,\mathcal{E}_1\big(U_1\mu,h\big)=\mathcal{E}_1\big(C_5\,U_1\mu,h\big).
\end{align*}
Hence
\begin{align*}
\mathcal{E}_1\big(C_5\,U_1\mu-U_1{\mu_g},h\big)\ge 0\quad\text{for all }h\in D\big(\mathcal{E}\big)\cap C_c(E)\text{ such that }h\ge 0.
\end{align*}
By \cite[Theo.~2.2.1]{FOT94} this implies that $C_5\,U_1\mu-U_1{\mu_g}\ge 0$. Thus $0\le U_1{\mu_g}\le C_5\,U_1\mu$ and therefore,
\begin{align*}
\Vert U_1{\mu_g}\Vert_{\scriptscriptstyle{L^\infty(E;\mu_{n,s,\varrho})}}\le C_5\,\Vert U_1{\mu}\Vert_{\scriptscriptstyle{L^\infty(E;\mu_{n,s,\varrho})}}<\infty.
\end{align*}
\end{proof}

Let $t>0$, $\mu\in \mathbf{S}$, $A\in\mathbf{A}^+_{c}$ and $f,h:E\to[0,\infty)$ measurable. Then we consider 
\begin{align}\label{equRevuz1}
\mathbb{E}_{h \mu_{n,s,\varrho}}\Big(\big(fA\big)_t\Big):=\int_{\mathbf{\Omega}}\int_{0}^tf\big(\mathbf{X}_\tau\big)\,dA_{\tau}\,d\mathbf{P}_{h\cdot\mu_{n,s,\varrho}}
\end{align}
and
\begin{multline}\label{equRevuz2}
\int_0^t\left\langle f\mu,p_\tau h\right\rangle\,d\tau:=\int_0^t\int_E \left(\widetilde{T_\tau h}\right)f\,d\mu\,d\tau\\
=\int_0^t\int_E\int_{\mathbf{\Omega}}h\big(\mathbf{X}_\tau(\omega)\big)\,d\mathbf{P}^{n,s,\varrho}_x(\omega)\,f(x)\,d\mu(x)\,d\tau,
\end{multline}
where for $\tau\ge 0$, $p_\tau h:=\widetilde{T_\tau h}$ denotes an $\mathcal{E}$-quasi continuous version of ${T_\tau h}$.

\begin{definition}[\emph{Revuz correspondence}]\label{defRevuz}
A measure $\mu\in \mathbf{S}$ and a AF $A\in\mathbf{A}^+_{c}$ are said to be in \emph{Revuz correspondence} if and only if equality of (\ref{equRevuz1}) and (\ref{equRevuz2}) holds for all $f,h:E\to[0,\infty)$ measurable.
\end{definition}

\begin{theorem}\label{theoRevuz}
Suppose Conditions \ref{conddensity}, \ref{condHamza} and \ref{condweakdiff} are satisfied. Then for $B\subset I$ the positive Radon measure $\mu_{\scriptscriptstyle{B}}:=\mu^{n,\varrho}_{\scriptscriptstyle{B}}:=\varrho\,\lambda^{\scriptscriptstyle{(n)}}_{\scriptscriptstyle{B}}$ is an element of $\mathbf{S}_{\scriptscriptstyle{00}}$ and in Revuz correspondence with the PCAF $\big(A_t^{\scriptscriptstyle{B}}\big)_{t\ge 0}$ given by 
\begin{align*}
A^{\scriptscriptstyle{B}}_t:=A^{n,s,\scriptscriptstyle{B}}_t:=\frac{1}{s^{n-\# B}}\int_{0}^t\mathbbm{1}_{E_+(B)}\big(\mathbf{X_\tau}\big)\,d\tau.
\end{align*}
\end{theorem}

\begin{proof}
Let $B\subset I$. In order to show that $\mu_{\scriptscriptstyle{B}}\in \mathbf{S}_{\scriptscriptstyle{00}}$ we first prove that $\mu^{n,s,\varrho}_{\scriptscriptstyle{B}}=\varrho\,\lambda^{n,s}_{\scriptscriptstyle{B}}\in \mathbf{S}_{\scriptscriptstyle{00}}$. Since $\mu^{n,s,\varrho}_{\scriptscriptstyle{B}}$ is a finite measure by Condition \ref{conddensity} one easily checks that $\mu^{n,s,\varrho}_{\scriptscriptstyle{B}}$ is of finite energy integral. It is left to show that the corresponding $1$-potential is essentially bounded. In the same way as in the proof of Proposition \ref{propstronglocal} we obtain that $\mathcal{E}_1\big(\mathbbm{1}_E,g\big)=\int_E g\,\mu_{n,s,\varrho}$ for all $g\in D(\mathcal{E})\cap C_c(E)$, i.e., $U_1\mu_{n,s,\varrho}=\mathbbm{1}_E\in D(\mathcal{E})$. Moreover, we have that 
\begin{multline*}
\mathcal{E}_1\big(U_1\mu_{n,s,\varrho},f\big)=\int_E f\,d\mu_{n,s,\varrho}=\sum_{B\subset I}\int_E f\,d\mu^{n,s,\varrho}_{\scriptscriptstyle{B}}\\
=\sum_{B\subset I}\mathcal{E}_1\big(U_1\mu^{n,s,\varrho}_{\scriptscriptstyle{B}},f\big)\quad\mbox{for all }f\in D\big(\mathcal{E}\big)\cap C_c(E),
\end{multline*}
where the existence of $U_1\mu^{n,s,\varrho}_{\scriptscriptstyle{B}}$, $B\subset I$, is due to Remark \ref{remUpot}, since $\mu^{n,s,\varrho}_{\scriptscriptstyle{B}}$ is of finite energy integral. The fact that $1$-potentials are uniquely determined yields $U_1\mu_{n,s,\varrho}=\sum_{B\subset I}U_1\Big(\mu^{n,s,\varrho}_{\scriptscriptstyle{B}}\Big)$. By \cite[Theo.~2.2.1]{FOT94} $1$-potentials are non-negative. Therefore, all measures $\mu^{n,s,\varrho}_{\scriptscriptstyle{B}}$, $B\subset I$, have essentially bounded $1$-potentials, since $U_1\mu_{{n,s,\varrho}}$ is essentially bounded. Applying Lemma \ref{lemS00} yields $\mu_{\scriptscriptstyle{B}}\in \mathbf{S}_{\scriptscriptstyle{00}}$.

Due to Lemma \ref{lemPACF} $\big(A^{\scriptscriptstyle{B}}_t\big)_{t\ge 0}$ is a PCAF. That $\mu_{\scriptscriptstyle{B}}$ and $\big(A^{\scriptscriptstyle{B}}_t\big)_{t\ge 0}$ are in Revuz correspondence follows by 
\begin{multline*}
\int_0^t\left\langle f\,\mu_{\scriptscriptstyle{B}},p_\tau h\right\rangle\,d\tau=\frac{1}{s^{n-\# B}}\int_0^t\left\langle f\,\mathbbm{1}_{E_+(B)}\mu_{n,s,\varrho},p_\tau h\right\rangle\,d\tau\\
=\frac{1}{s^{n-\# B}}\int_0^t\int_E \left(\widetilde{T_\tau h}\right)f\mathbbm{1}_{E_+(B)}\,d\mu_{n,s,\varrho}\,d\tau=\frac{1}{s^{n-\# B}}\int_0^t\int_E h\,\left(\widetilde{T_\tau\big(f\mathbbm{1}_{E_+(B)}\big)}\right)\,d\mu_{n,s,\varrho}\,d\tau\\
=\frac{1}{s^{n-\# B}}\int_0^t\int_E\int_{\mathbf{\Omega}}\big(f\mathbbm{1}_{E_+(B)}\big)\big(\mathbf{X}_\tau(\omega)\big)\,d\mathbf{P}^{n,s,\varrho}_x(\omega)\,h(x)\,d\mu_{n,s,\varrho}(x)\,d\tau\\
=\frac{1}{s^{n-\# B}}\int_0^t\int_E\int_{\mathbf{\Omega}}f\big(\mathbf{X}_\tau(\omega)\big)\,\mathbbm{1}_{E_+(B)}\big(\mathbf{X}_\tau(\omega)\big)\,d\mathbf{P}^{n,s,\varrho}_x(\omega)\,h(x)\,d\mu_{n,s,\varrho}(x)\,d\tau\\
=\int_0^t\int_{\mathbf{\Omega}}f\big(\mathbf{X}_\tau\big)\,\frac{1}{s^{n-\# B}}\mathbbm{1}_{E_+(B)}\big(\mathbf{X}_\tau\big)\,d\mathbf{P}_{h\cdot\mu_{n,s,\varrho}}\,d\tau
=\int_{\mathbf{\Omega}}\int_0^tf\big(\mathbf{X}_\tau\big)\,\frac{1}{s^{n-\# B}}\mathbbm{1}_{E_+(B)}\big(\mathbf{X}_\tau\big)\,d\tau\,d\mathbf{P}_{h\cdot\mu_{n,s,\varrho}}\\
=\int_{\mathbf{\Omega}}\int_0^tf\big(\mathbf{X}_\tau\big)\,dA^{\scriptscriptstyle{B}}_{\tau}\,d\mathbf{P}_{h\cdot\mu_{n,s,\varrho}}=\mathbb{E}_{h\mu_{n,s,\varrho}}\Big(\big(fA^{\scriptscriptstyle{B}}\big)_t\Big)
\end{multline*}
for $f,h:E\to[0,\infty)$ measurable.
\end{proof}

\begin{proposition}\label{propS00ex}
Suppose Conditions \ref{conddensity} and \ref{condHamza}. Let $\eta\in C^0(E)$ such that ${\eta}\ge 0$.
\begin{enumerate}
\item[(i)]
If $\mu\in\mathbf{S}_{\scriptscriptstyle{00}}$ and $A_t=\int_0^t g\big(\mathbf{X}_\tau\big)\mathbbm{1}_M\big(\mathbf{X}_{\tau}\big)\,d\tau$, $t\ge 0$, as in Lemma \ref{lemPACF}, are in Revuz correspondence, then
\begin{align*}
{\mu_{\eta}}:={\eta}\mu\quad\text{and}\quad{A}^{\eta}_t:=\int_0^t {\eta}\big(\mathbf{X}_\tau\big)\,g\big(\mathbf{X}_\tau\big)\mathbbm{1}_M\big(\mathbf{X}_{\tau}\big)\,d\tau
\end{align*}
are in Revuz correspondence.
\item[(ii)]
If, moreover, ${\eta}$ has compact support, then ${\mu_{\eta}}\in\mathbf{S}_{\scriptscriptstyle{00}}$.  
\end{enumerate}
\end{proposition}
\begin{proof}
\begin{enumerate}
\item[(i)]
Since $\mu\in\mathbf{S}_{\scriptscriptstyle{00}}\subset\mathbf{S}$, there exists a generalized compact nest $\big(F_k\big)_{k\in\mathbb{N}}$ such that $\mu\Big(E\setminus\bigcup_{n=1}^\infty F_k\Big)=0$ and $\mathbbm{1}_{F_k}\mu\in\mathbf{S}_{\scriptscriptstyle{00}}$ for each $k\in\mathbb{N}$ according to \cite[Theo.~2.2.4]{FOT94}. Hence $E\setminus\bigcup_{n=1}^\infty F_k$ is also a ${\mu_{\eta}}$-null set. Moreover, due to Lemma \ref{lemS00} we have $\mathbbm{1}_{F_k}{\mu_{\eta}}=\big(\mathbbm{1}_{F_k}{\eta}\big)\mu\in\mathbf{S}_{\scriptscriptstyle{00}}$. Another application of \cite[Theo.~2.2.4]{FOT94} yields ${\mu_{\eta}}\in\mathbf{S}$. Let $f,h:E\to[0,\infty)$ measurable. We have that
\begin{multline}\label{equRc1}
\mathbb{E}_{h\mu^{n,s,\varrho}}\big((f{A}^{\eta})_t\big)=\int_{\mathbf{\Omega}}\int_0 ^t f\big(\mathbf{X}_{\tau}\big)\,d{A}^{\eta}_\tau\,d\mathbf{P}_{h\mu_{n,s,\varrho}}\\
=\int_{\mathbf{\Omega}}\int_0 ^t \eta\big(\mathbf{X}_{\tau}\big)\,f\big(\mathbf{X}_{\tau}\big)\,d{A}_\tau\,d\mathbf{P}_{h\mu_{n,s,\varrho}}=\mathbb{E}_{h\mu^{n,s,\varrho}}\Big(\big((\eta f){A}\big)_t\Big).
\end{multline}
On the other hand,
\begin{align}\label{equRc2}
\int_0^t\left\langle f\big(\eta\mu\big),p_\tau h\right\rangle\,d\tau=\int_0^t\int_E p_\tau h\,f\,\eta\,d\mu\,d\tau=\int_0^t\left\langle \big(\eta f\big)\mu,p_\tau h\right\rangle\,d\tau.
\end{align}
Since $\eta f:E\to[0,\infty)$ is measurable and $\mu$ is in Revuz correspondence with $\big(A_t\big)_{t\ge 0}$ we obtain that (\ref{equRc1}) equals (\ref{equRc2}) and the desired statement follows.
\item[(ii)]
This follows by Lemma \ref{lemS00}.
\end{enumerate}
\end{proof}

\begin{remark}\label{remS00plus}
If $\mu_1,\mu_2\in\mathbf{S}_{\scriptscriptstyle{00}}$ with Revuz corresponding AFs $A_1$, $A_2$, respectively. Then $\mu_1+\mu_2\in\mathbf{S}_{\scriptscriptstyle{00}}$ with Revuz corresponding AF $A$ given by $A:=A_1+A_2$.
\end{remark}

\begin{theorem}\label{theorep1}
Suppose Conditions \ref{conddensity}, \ref{condHamza} and \ref{condweakdiff} are satisfied. Let $f\in\mathcal{D}$. Then
\begin{align}\label{equdecomp}
f\big(\mathbf{X}_t\big)-f\big(\mathbf{X}_0\big)=\mathbf{M}_t^{[f]}-\mathbf{N}_t^{[f]}\quad\mathbf{P}^{n,s,\varrho}_x-\text{a.s.~for q.e. }x\in E,
\end{align}
where $\mathbf{M}_t^{[f]}$ is a MAF with quadratic variation
\begin{align*}
\left\langle \mathbf{M}^{[f]} \right\rangle_t=2\sum_{\varnothing\not=B\subset I}\int_0^t\sum_{i\in B}\big(\partial_if\big)^2\big(\mathbf{X}_\tau\big)\,\mathbbm{1}_{{E_+(B)}}\big(\mathbf{X}_\tau\big)\,d\tau
\end{align*}
and
\begin{multline*}
\mathbf{N}^{[f]}_t=\int_0^t\Big(\sum_{B\subset I}\Big(\sum_{i\in B}\Big(\partial^2_i f+\partial_i f\partial_i \ln(\varrho)\Big)\Big)\big(\mathbf{X}_\tau\big)\,\mathbbm{1}_{{E_+(B)}}\big(\mathbf{X}_\tau\big)\Big)\\
+\Big(\sum_{B\subset I}\Big(\frac{1}{s}\sum_{i\in I\setminus B}\partial_i f\Big)(\mathbf{X}_\tau\big)\,\mathbbm{1}_{{E_+(B)}}\big(\mathbf{X}_\tau\big)\Big)\,d\tau.
\end{multline*}
\end{theorem}

\begin{remark}
Note that the decomposition (\ref{equdecomp}) is valid $\mathbf{P}^{n,s,\varrho}_x-\text{a.s.~for q.e. }x\in E$. This is weaker then the statement in \cite[Theo.~5.2.5]{FOT94} where the decomposition holds $\mathbf{P}^{n,s,\varrho}_x-\text{a.s.~for each }x\in E$. This is caused by the fact that in our setting we do not know if the \emph{absolute continuity condition} is fulfilled.  
\end{remark}

\begin{proof}
We have to check the assumptions of \cite[Theo.~5.2.5]{FOT94}. $f\in\mathcal{D}\subset D\big(\mathcal{E}\big)$ is clearly bounded and continuous. The measure $\nu_{\scriptscriptstyle{\langle f\rangle}}\in \mathbf{S}_{\scriptscriptstyle{00}}$ due to Proposition \ref{propenergymeasure}, Theorem \ref{theoRevuz}, Proposition \ref{propS00ex}(ii) and Remark \ref{remS00plus} applied inductively. In addition, these results yield that $\nu_{\scriptscriptstyle{\langle f\rangle}}$ is in Revuz correspondence with the PCAF
\begin{align*}
2\sum_{\varnothing\not=B\subset I}\int_0^t\sum_{i\in B}\big(\partial_if\big)^2\big(\mathbf{X}_\tau\big)\,\mathbbm{1}_{{E_+(B)}}\big(\mathbf{X}_\tau\big)\,d\tau.
\end{align*}
By Proposition \ref{propmeasureform} 
\begin{align*}
\mathcal{E}\big(f,g\big)=\big\langle\nu_f,g\big\rangle=\int_{E}g\,d\nu_f
\end{align*}
with 
\begin{align*}
\nu_f=\sum_{B\subset I}\Bigg(\sum_{i\in B}\Big(-\partial^2_i f-\partial_i f\partial_i \ln(\varrho)\Big)\Bigg)\,\lambda^{n,s,\varrho}_{\scriptscriptstyle{B}}-\frac{1}{s}\sum_{B\subset I}\Big(\sum_{i\in I\setminus B}\partial_i f\Big)\,\lambda^{n,s,\varrho}_{\scriptscriptstyle{B}}
\end{align*}
for all $f,g\in\mathcal{D}$. We can split the densities contained in $\nu_{f}$ into positive and negative part. This yields two positive Radon measures $\nu_{f}^{+}$ and $\nu_{f}^{-}$ such that $\nu_{f}=\nu_{f}^{+}-\nu_{f}^{-}$. These measures belong to $\mathbf{S}_{\scriptscriptstyle{00}}$ by Theorem \ref{theoRevuz}, Proposition \ref{propS00ex} and Remark \ref{remS00plus}. We can calculate the associated PCAFs $A^{\scriptscriptstyle{+}}$ and $A^{\scriptscriptstyle{-}}$ in the same way like in the case of $\nu_{\scriptscriptstyle{\langle f\rangle}}$. By \cite[Theo.~5.2.5]{FOT94} $\mathbf{N}^{[f]}_t=-A^{\scriptscriptstyle{+}}+A^{\scriptscriptstyle{-}}$ and we obtain that
\begin{multline*}
\mathbf{N}^{[f]}_t=\int_0^t\Big(\sum_{B\subset I}\Big(\sum_{i\in B}\Big(\partial^2_i f+\partial_i f\partial_i \ln(\varrho)\Big)\Big)\big(\mathbf{X}_\tau\big)\,\mathbbm{1}_{{E_+(B)}}\big(\mathbf{X}_\tau\big)\Big)\\
+\Big(\sum_{B\subset I}\Big(\frac{1}{s}\sum_{i\in I\setminus B}\partial_i f\Big)(\mathbf{X}_\tau\big)\,\mathbbm{1}_{{E_+(B)}}\big(\mathbf{X}_\tau\big)\Big)\,d\tau.
\end{multline*}
\end{proof}

\begin{corollary}\label{coroskoro}
Let $j\in I$. We denote by $\pi_j:\mathbb{R}^n\to\mathbb{R}$, $x\mapsto x_j$, the projection on the $j$-th coordinate. Then under the assumptions of Theorem \ref{theorep1} the process $\mathbf{M}$ is characterized via its coordinate processes $\big(\mathbf{X}_t^j\big)_{t\ge 0}:=\big(\pi_j(\mathbf{X}_{t})\big)_{t\ge 0}$, $1\le j\le n$, by
\begin{multline}\label{rep}
\mathbf{X}_{t}^j-\mathbf{X}_0^j=\mathbbm{1}_{\mathring{E}}\big(\mathbf{X}_t\big)\,\sqrt{2}\,B^j_t+\int_{0}^{t}\partial_j\ln(\varrho)\big(\mathbf{X}_\tau\big)\mathbbm{1}_{\mathring{E}}\big(\mathbf{X}_\tau\big)\,d\tau\\
+\sum_{\varnothing\not=B\subsetneq I}\left\{\begin{array}{ll}
  \mathbbm{1}_{{E_+(B)}}\big(\mathbf{X}_t\big)\,\sqrt{2}\,B^j_t+\int_0^{t} \partial_j\ln(\varrho)\big(\mathbf{X}_\tau\big)\mathbbm{1}_{{E_+(B)}}\big(\mathbf{X}_\tau\big)\,d\tau, & \text{if }j\in B\\
  \frac{1}{s}\,\int_0^{t}\mathbbm{1}_{{E_+(B)}}\big(\mathbf{X}_\tau\big)\,d\tau, & \text{if }j\in I\setminus B
  \end{array}\right.\\
+\frac{1}{s}\int_0^{t}\mathbbm{1}_{\{(0,\ldots,0)\}}\big(\mathbf{X}_\tau\big)\,d\tau,
\end{multline}
where $(B^j_t)_{t\ge 0}$ is a one dimensional Brownian motion with $B^j_0=0$.
\end{corollary}

\begin{proof}
We consider
\begin{align*}
\pi_j^k(x):=
\left\{\begin{array}{ll}
  x_j, & \text{if }x\in [0,k+1)^n\\
  0, & \text{if }x\in [k+2,\infty)^n
  \end{array}\right., \quad 1 \le j \le n, \, k \in\mathbb{N},\text{ such that }\pi_j^k\in\mathcal{D}.
\end{align*}
Furthermore, we define
\begin{align*}
\tau_k:=\inf\big\{t \ge 0 \,|\, \mathbf{X}_t \not\in [0,k]^n\big\}, \quad k \in {\mathbb N}. 
\end{align*}
$(\tau_k)_{k\in\mathbb{N}}$ is a sequence of stopping times with $\tau_k\uparrow\infty$ as $k\to\infty$. Now using the decomposition (\ref{equdecomp}) we obtain for $k\in\mathbb{N}$ and $j\in I$ the representation
\begin{multline*}
\mathbf{X}_{t\wedge\tau_k}^j-\mathbf{X}_0^j=\pi_j^k\big(\mathbf{X}_{t\wedge\tau_k}\big)-\pi_j^k\big(\mathbf{X}_0\big)=\mathbf{M}_{t\wedge\tau_k}^{[\pi_j^k]}+\mathbf{N}_{t\wedge\tau_k}^{[\pi_j^k]}\\
=\mathbf{M}_{t\wedge\tau_k}^{[\pi^k_j]}+\int_{0}^{t\wedge\tau_k}\partial_j\ln(\varrho)\big(\mathbf{X}_\tau\big)\mathbbm{1}_{\mathring{E}}\big(\mathbf{X}_\tau\big)\,d\tau\\
+\sum_{\varnothing\not=B\subsetneq I}\left\{\begin{array}{ll}
  \int_0^{t\wedge\tau_k} \partial_j\varrho\big(\mathbf{X}_\tau\big)\mathbbm{1}_{{E_+(B)}}\big(\mathbf{X}_\tau\big)\,d\tau, & \text{if }j\in B\\
  \frac{1}{s}\,\int_0^{t\wedge\tau_k}\mathbbm{1}_{{E_+(B)}}\big(\mathbf{X}_\tau\big)\,d\tau, & \text{if }j\in I\setminus B
  \end{array}\right.\\
+\frac{1}{s}\int_0^{t\wedge\tau_k}\mathbbm{1}_{\{(0,\ldots,0)\}}\big(\mathbf{X}_\tau\big)\,d\tau.
\end{multline*}
 Additionally we have that the quadratic variation
\begin{align*}
\left\langle \mathbf{M}^{[\pi^k_j]}\right\rangle_{t\wedge\tau_k}=2\,\big(t\wedge\tau_k\big)
\end{align*}
as long as $\mathbf{X}_{t\wedge\tau_k}$, $t\ge 0$, takes values in $E_+(B)$, $\varnothing\not=B\subset I$.
Hence in these situations for $k\in\mathbb{N}$ large enough $\mathbf{M}_{t\wedge\tau_k}^{[\pi^k_j]}=\mathbf{M}_{t}^{[\pi_j]}$ is a continuous, local martingale with quadratic variation $2t$. Thus by L\'{e}vy's theorem (see e.g.~\cite[Sect.~3.4,~Theo.~(4.1)]{Dur96}) together with the scaling property of Brownian motion we obtain
\begin{align*}
\mathbf{M}_{t}^{[\pi_j]}=B^j_t,\quad {t\ge 0},
\end{align*}
as long as $\mathbf{X}_{t}$, $t\ge 0$, takes values in $E_+(B)$, $\varnothing\not=B\subset I$, where $(B^j_t)_{t\ge 0}$ is a one dimensional Brownian motion with $B^j_0=0$. For $j\in I$ this yields the representation (\ref{rep}).
\end{proof}

\begin{remark}\label{remBM}
For $\varnothing\not=B\subsetneq I$ and $j,k\in B$ we obtain that $\left\langle \mathbf{M}^{[\pi_j]},\mathbf{M}^{[\pi_k]}\right\rangle_{t}=2\delta_{jk}t$, where $\delta_{jk}$ denotes the Kronecker delta. Therefore, on the boundary parts $E_{\scriptscriptstyle{+}}(B)$ of $\partial E$ the constructed process is a $(\#B)$-dimensional Brownian motion scaled by $\sqrt{2}$ when it spends time there. 
\end{remark}

\section{Application to the dynamical wetting model in $(d+1)$-dimension}

Let $d\in\mathbb{N}$ and $D_{\scriptscriptstyle{d}}:=(0,1]^d\subset\mathbb{R}^d$. For $N\in\mathbb{N}$ we define $D_{\scriptscriptstyle{d,N}}:=ND_{\scriptscriptstyle{d}}\cap\mathbb{Z}^d$, where $ND_{\scriptscriptstyle{d}}:=\big\{N\theta\,\big|\,\theta\in D_{\scriptscriptstyle{d}}\big\}$. Here $N$ stands for the scaling parameter. The discretized set $D_{d,\scriptscriptstyle{N}}$ is a \emph{microscopic} correspondence to the \emph{macroscopic} domain $D_{\scriptscriptstyle{d}}$ and given by $D_{\scriptscriptstyle{d,N}}=\big\{1,2,\ldots,N\big\}^d$. The boundary $\partial D_{\scriptscriptstyle{d,N}}$ of $D_{\scriptscriptstyle{d,N}}$ is defined by $\partial D_{\scriptscriptstyle{d,N}}:=\big\{x\not\in D_{\scriptscriptstyle{d,N}}\,\big|\,|x-y|_{\text{euc}}=1\mbox{ for some }y\in D_{\scriptscriptstyle{d,N}}\big\}$ and the closure $\overline{D_{d,\scriptscriptstyle{N}}}$ of $D_{\scriptscriptstyle{d,N}}$ is given by $\overline{D_{d,\scriptscriptstyle{N}}}:=D_{\scriptscriptstyle{d,N}}\cup\partial D_{d,\scriptscriptstyle{N}}$. Hence $\overline{D_{\scriptscriptstyle{d,N}}}=\big\{0,1,2,\ldots,N+1\big\}^d$. For fixed $N\in\mathbb{N}$ we consider the \emph{space of interfaces}
\begin{align*}
\overline{\Omega^+_{\scriptscriptstyle{d,N}}}:=[0,\infty)^{{D_{\scriptscriptstyle{d,N}}}}:=\Big\{\phi:{D_{\scriptscriptstyle{d,N}}}\to[0,\infty)\Big\}=\Big\{\phi:=(\phi_x)_{x\in {D_{\scriptscriptstyle{d,N}}}}\subset [0,\infty)^{N^d}\Big\}
\end{align*}
on ${D_{\scriptscriptstyle{d,N}}}$. Note that the variable $\phi_x$, $x\in {D_{\scriptscriptstyle{d,N}}}$, describes the height of an interface at position $x\in {D_{\scriptscriptstyle{d,N}}}$ measured with respect to the reference hyperplane ${D_{\scriptscriptstyle{d,N}}}$. Therefore, $\phi_x$, $x\in {D_{\scriptscriptstyle{d,N}}}$, is also called \emph{height variable}. We extend $\phi\in\overline{\Omega_{\scriptscriptstyle{d,N}}^+}$ to the boundary $\partial D_{\scriptscriptstyle{d,N}}$ by setting $\phi_x=0$ for all $x\in\partial D_{\scriptscriptstyle{d,N}}$. The restriction for the functions $\phi$ to take values in $[0,\infty)\subset\mathbb{R}$ reflects the fact that a hard wall is settled at height level $0$ of the interface. 

The potential energy of an interface $\phi\in\overline{\Omega_{\scriptscriptstyle{d,N}}^+}$ is given by a \emph{Hamiltonian} with \emph{zero boundary condition}, i.e.,
\begin{align}\label{equHam}
\overline{\Omega_{\scriptscriptstyle{d,N}}^+}\ni\phi\mapsto H^{\scriptscriptstyle{V}}_{\scriptscriptstyle{d,N}}(\phi):=\frac{1}{2}\sum_{\stackunder{|x-y|_{\text{euc}}=1}{x,y\in{{\overline{D_{\scriptscriptstyle{d,N}}}}}}}V\big(\phi_x-\phi_y\big)\in\mathbb{R},
\end{align}
where the pair interaction potential $V$ fulfills Condition \ref{condpotential} below.
\begin{condition}\label{condpotential}
$V:\mathbb{R}\to[-b,\infty)$, $b\in[0,\infty)$, is continuously differentiable and symmetric, i.e., $V(-r)=V(r)$ for all $r\in\mathbb{R}$. Moreover, the following integrability conditions are satisfied:
\begin{enumerate}
\item[(i)]
$\quad\kappa:=\int_{\mathbb{R}}\exp\big(-V(r)\big)\,dr<\infty$;
\item[(ii)]
$\mathbb{V}'(x,\cdot)\in L^2\big(\overline{\Omega_{\scriptscriptstyle{d,N}}^+};\mu_{n,s,\varrho}\big)$ for all $x\in {D_{\scriptscriptstyle{d,N}}}$, where
\begin{align*}
\overline{\Omega_{\scriptscriptstyle{d,N}}^+}\ni\phi\mapsto \mathbb{V}'(x,\phi):=\sum_{\stackunder{|x-y|_{\text{euc}}=1}{y\in{{\overline{D_{\scriptscriptstyle{d,N}}}}}}}V'\big(\phi_x-\phi_y\big)\in\mathbb{R}.
\end{align*}
\end{enumerate}
\end{condition}

\begin{remark}\label{remcondapp}
Condition \ref{condpotential} guarantees that $V(0)\in[-b,\infty)$, hence flat interfaces are natural elements in the space of interfaces $\overline{\Omega_{\scriptscriptstyle{d,N}}^+}$, i.e., occur with positive probability, see (\ref{repmeasure}) below. Furthermore, Condition \ref{condpotential} implies Conditions \ref{condHamza} and \ref{condweakdiff} (see Remark \ref{remcondequi}).
\end{remark}
A natural distribution on the space of interfaces $\Big(\overline{\Omega_{\scriptscriptstyle{d,N}}^+},\mathcal{B}\big(\overline{\Omega_{\scriptscriptstyle{d,N}}^+}\big)\Big)$ is given by the probability measure $\mu_{\scriptscriptstyle{d,N}}^{\scriptscriptstyle{V,s}}$ defined by
\begin{align}\label{repmeasure}
d\mu_{\scriptscriptstyle{d,N}}^{\scriptscriptstyle{V,s}}(\phi)=\frac{1}{Z_{\scriptscriptstyle{d,N}}^{\scriptscriptstyle{V,s}}}\exp\Big(-H^{\scriptscriptstyle{V}}_{\scriptscriptstyle{d,N}}(\phi)\Big)\,\prod_{x\in {D_{\scriptscriptstyle{d,N}}}}\Big(s\,d\delta_0^x+d\phi_+^x\Big),\quad\phi\in\overline{\Omega_{\scriptscriptstyle{d,N}}^+},
\end{align}
with pair interaction potential $V$ under Condition \ref{condpotential} and normalizing constant $Z_{\scriptscriptstyle{d,N}}^{\scriptscriptstyle{V,s}}$. $\mu_{\scriptscriptstyle{d,N}}^{\scriptscriptstyle{V,s}}$ is a \emph{finite volume Gibbs measure} conditioned on $[0,\infty)^{{D_{\scriptscriptstyle{d,N}}}}$.
The corresponding space of square integrable functions we denote by $L^2\Big(\overline{\Omega_{\scriptscriptstyle{d,N}}^+};\mu_{\scriptscriptstyle{d,N}}^{\scriptscriptstyle{V,s}}\Big)$. Next we define the probability density
\begin{align*}
\varrho(\phi):=\varrho^{\scriptscriptstyle{V,s}}_{\scriptscriptstyle{d,N}}(\phi):=\frac{1}{Z_{\scriptscriptstyle{d,N}}^{\scriptscriptstyle{V,s}}}\exp\Big(-H^{\scriptscriptstyle{V}}_{\scriptscriptstyle{d,N}}(\phi)\Big),\quad\phi\in\overline{\Omega_{\scriptscriptstyle{d,N}}^+}.
\end{align*}
Hence we can rewrite (\ref{repmeasure}) as
\begin{multline*}
d\mu_{\scriptscriptstyle{N^d,s,\varrho}}:=d\mu_{\scriptscriptstyle{d,N}}^{\scriptscriptstyle{V,s}}=\varrho\,\prod_{x\in {D_{\scriptscriptstyle{d,N}}}}\Big(s\,d\delta_0^x+d\phi^x_+\Big)\\
=\varrho\,\sum_{B\subset D_{\scriptscriptstyle{d,N}}}s^{N^d-\#B}\left(\prod_{x\in B}d\phi^x_+\prod_{y\in D_{\scriptscriptstyle{d,N}}\setminus B}d\delta^y_{0}\right)
=\varrho\,\sum_{B\subset D_{\scriptscriptstyle{d,N}}}d\lambda^{N^d,s}_{\scriptscriptstyle{B}}
=\varrho\,dm_{N^d,s},\quad\phi\in\overline{\Omega_{\scriptscriptstyle{d,N}}^+}.
\end{multline*}

For each $\phi\in\overline{\Omega_{\scriptscriptstyle{d,N}}^+}$ we denote by
\begin{align*}
D^{\scriptscriptstyle{\text{dry}}}_{\scriptscriptstyle{d,N}}(\phi):=\big\{x\in D_{\scriptscriptstyle{d,N}}\,\big|\phi_x=0\big\}\quad\mbox{and}\quad D^{\scriptscriptstyle{\text{wet}}}_{\scriptscriptstyle{d,N}}(\phi):=\big\{x\in D_{\scriptscriptstyle{d,N}}\,\big|\phi_x>0\big\},
\end{align*}
\emph{dry regions} and \emph{wet regions} associated with the interface $\phi$, respectively, and define for $A,B\subset D_{\scriptscriptstyle{d,N}}$,
\begin{align*}
\Omega_{\scriptscriptstyle{d,N,A}}^{+\scriptscriptstyle{\text{,dry}}}:=\Big\{\phi\in\overline{\Omega_{\scriptscriptstyle{d,N}}^+}\,\Big|\,D^{\scriptscriptstyle{\text{dry}}}_{\scriptscriptstyle{d,N}}(\phi)=A\Big\}\quad\mbox{and}\quad\Omega_{\scriptscriptstyle{d,N,B}}^{+\scriptscriptstyle{\text{,wet}}}:=\Big\{\phi\in\overline{\Omega_{\scriptscriptstyle{d,N}}^+}\,\Big|\,D^{\scriptscriptstyle{\text{wet}}}_{\scriptscriptstyle{d,N}}(\phi)=B\Big\},
\end{align*}
respectively.
\begin{remark}
The following decomposition of the state space is valid:
\begin{align*}
\overline{\Omega_{\scriptscriptstyle{d,N}}^+}=\dot{\bigcup}_{A\subset D_{\scriptscriptstyle{d,N}}}\Omega_{\scriptscriptstyle{d,N,A}}^{+\scriptscriptstyle{\text{,dry}}}=\dot{\bigcup}_{B\subset D_{\scriptscriptstyle{d,N}}}\Omega_{\scriptscriptstyle{d,N,B}}^{+\scriptscriptstyle{\text{,wet}}}.
\end{align*}
\end{remark}
Therefore, $\mu_{\scriptscriptstyle{N^d,s,\varrho}}=\sum_{B\subset D_{\scriptscriptstyle{d,N}}}\mu^{\scriptscriptstyle{N^d,s,\varrho}}_{\scriptscriptstyle{B}}$ with $\mu^{\scriptscriptstyle{N^d,s,\varrho}}_{\scriptscriptstyle{B}}:=\mu_{\scriptscriptstyle{N^d,s,\varrho}}\Big|_{\Omega_{\scriptscriptstyle{d,N,B}}^{+\scriptscriptstyle{\text{,wet}}}}$.

\begin{theorem}\label{theosumdiriapp}
Let $d,N\in\mathbb{N}$. For $s\in(0,\infty)$ we have that under Condition \ref{condpotential} 
\begin{align}\label{formapp}
\mathcal{E}^{N^d,s,\varrho}\big(F,G\big):=\sum_{\varnothing\not=B\subset D_{\scriptscriptstyle{d,N}}}\mathcal{E}^{N^d,s,\varrho}_{\scriptscriptstyle{B}}\big(F,G\big),\quad F,G\in \mathcal{D}=C_c^2\big(\overline{\Omega_{\scriptscriptstyle{d,N}}^+}\big)
\end{align}
with 
\begin{align*}
\mathcal{E}^{N^d,s,\varrho}_{\scriptscriptstyle{B}}\big(F,G\big):=\sum_{x\in B}\int_{\Omega_{\scriptscriptstyle{d,N,B}}^{+\scriptscriptstyle{\text{,wet}}}}\partial_xF\,\partial_x G\,d\mu^{\scriptscriptstyle{N^d,s,\varrho}}_{\scriptscriptstyle{B}},\quad\varnothing\not=B\subset D_{\scriptscriptstyle{d,N}},
\end{align*}
is a densely defined, positive definite, symmetric bilinear form, which is closable on\\ $L^2\big(\overline{\Omega_{\scriptscriptstyle{d,N}}^+};\mu_{\scriptscriptstyle{N^d,s,\varrho}}\big)$. Its closure $\big(\mathcal{E}^{N^d,s,\varrho},D(\mathcal{E}^{N^d,s,\varrho})\big)$ is a conservative, strongly local, regular, symmetric Dirichlet form on $L^2\big(\overline{\Omega_{\scriptscriptstyle{d,N}}^+};\mu_{\scriptscriptstyle{N^d,s,\varrho}}\big)$. 
\end{theorem}

\begin{remark}
Note that for functions in $\mathcal{D}$, $l\in\{1,2\}$ and $x\in D_{\scriptscriptstyle{d,N}}$ we denote by $\partial_x^l$ the partial derivative of order $l$ with respect to the variable $\phi_x$. In particular, $\partial_x:=\partial_x^1$.
\end{remark}

\begin{proof}
Use Remark \ref{remcondapp} and apply Theorem \ref{theosumdiri}.
\end{proof}

\begin{proposition}\label{propgenapp}
Suppose that 
Condition \ref{condpotential} is satisfied. There exists a unique, positive, self-adjoint, linear operator $\big(\mathcal{H}^{N^d,s,\varrho},D(\mathcal{H}^{N^d,s,\varrho})\big)$ on $L^2\big(\overline{\Omega_{\scriptscriptstyle{d,N}}^+};\mu_{N^d,s,\varrho}\big)$ such that
\begin{multline*}
D(\mathcal{H}^{N^d,s,\varrho})\subset D(\mathcal{E}^{N^d,s,\varrho})\quad\text{and}\quad\mathcal{E}^{N^d,s,\varrho}\big(F,G\big)=\Big(\mathcal{H}^{N^d,s,\varrho} F,G\Big)_{L^2(\overline{\Omega_{\scriptscriptstyle{d,N}}^+};\mu_{N^d,s,\varrho})}\\
\quad\text{for all }F\in D(\mathcal{H}^{N^d,s,\varrho}),~G\in D(\mathcal{E}^{N^d,s,\varrho}).
\end{multline*}
\end{proposition}
\begin{proof}
Use Remark \ref{remcondapp} and apply Proposition \ref{propgen}.
\end{proof}

For $F\in\mathcal{D}$ we define
\begin{align*}
\mathcal{L}^{N^d,s,\varrho}{F}:=\sum_{x\in D_{\scriptscriptstyle{d,N}}}\partial^2_x F+\sum_{x\in D_{\scriptscriptstyle{d,N}}}\partial_x F\,\partial_x (\ln\varrho)+\frac{1}{s}\sum_{x\in D_{\scriptscriptstyle{d,N}}}\partial_x F
\end{align*}
and
\begin{align*}
\mathcal{L}^{N^d,{\scriptscriptstyle{\infty}},\varrho} F:=\sum_{x\in D_{\scriptscriptstyle{d,N}}}\partial^2_x F+\sum_{x\in D_{\scriptscriptstyle{d,N}}}\partial_x F\,\partial_x (\ln\varrho).
\end{align*}
\begin{proposition}\label{propibpapp}
Suppose Condition \ref{condpotential} is satisfied. For functions 
\begin{align*}
F,G\in \mathcal{D}_{\scriptscriptstyle{\text{Wentzell}}}:=\Big\{H\in\mathcal{D}\,\Big|\,\big(\mathcal{L}^{N^d,s,\varrho}H\big)\big|_{\Omega_{\scriptscriptstyle{d,N,B}}^{+\scriptscriptstyle{\text{,wet}}}}=0\text{ for all }B\subsetneq D_{\scriptscriptstyle{d,N}}\Big\}
\end{align*}
we have the representation $\mathcal{E}^{N^d,s,\varrho}\big(F,G\big)=\Big(-\mathcal{L}^{N^d,{\scriptscriptstyle{\infty}},\varrho} F,G\Big)_{L^2(\overline{\Omega_{\scriptscriptstyle{d,N}}^+};\mu_{N^d,s,\varrho})}$.
\end{proposition}

\begin{proof}
Use Remark \ref{remcondapp} and apply Proposition \ref{propibp}.
\end{proof}

\begin{remark}
Elements from $\mathcal{D}_{\scriptscriptstyle{\text{Wentzell}}}$ are said to fulfill a \emph{Wentzell type boundary condition}.
\end{remark}

\begin{theorem}\label{theoprocessapp}
Suppose that Condition \ref{condpotential} is satisfied. Then there exists a conservative diffusion process (i.e.~a strong Markov process with continuous sample paths and infinite life time)
\begin{align*}
\mathbf{M}^{N^d,s,\varrho}=\left(\mathbf{\Omega},\mathbf{F},(\mathbf{F}_t)_{t\ge 0},(\mathbf{X}_t)_{t\ge 0},(\mathbf{\Theta}_t)_{t\ge 0},(\mathbf{P}^{N^d,s,\varrho}_{\phi})_{\phi\in \overline{\Omega_{\scriptscriptstyle{d,N}}^+}}\right)
\end{align*}
with state space $\overline{\Omega_{\scriptscriptstyle{d,N}}^+}$ which is properly associated with $\big(\mathcal{E}^{N^d,s,\varrho},D(\mathcal{E}^{N^d,s,\varrho})\big)$.
$\mathbf{M}^{N^d,s,\varrho}$ is up to $\mu_{N^d,s,\varrho}$-equivalence unique. In particular, $\mathbf{M}^{N^d,s,\varrho}$ is $\mu_{N^d,s,\varrho}$-symmetric and has $\mu_{N^d,s,\varrho}$ as invariant measure.
\end{theorem}
In the above theorem $\mathbf{M}^{N^d,s,\varrho}$ is canonical, i.e., $\mathbf{\Omega}=C^0\big([0,\infty),\overline{\Omega_{\scriptscriptstyle{d,N}}^+}\big)$, the space of continuous functions on $[0,\infty)$ into $\overline{\Omega_{\scriptscriptstyle{d,N}}^+}$, $\mathbf{X}_t(\omega)=\omega(t)$, $\omega\in\mathbf{\Omega}$. The filtration $(\mathbf{F}_t)_{t\ge 0}$ is the natural minimum completed admissible filtration obtained from the $\sigma$-algebras $\mathbf{F}^0_t:=\sigma\Big\{\mathbf{X}_{\tau}\,\Big|\,0\le \tau\le t\Big\}$, $t\ge 0$, and $\mathbf{F}:=\mathbf{F}_{\scriptscriptstyle{\infty}}:=\bigvee_{t\in[0,\infty)}\mathbf{F}_t$.

\begin{proof}
Use Remark \ref{remcondapp} and apply Theorem \ref{theoprocess}.
\end{proof}

\begin{theorem}\label{theomartingaleapp}
The diffusion process $\mathbf{M}^{N^d,s,\varrho}$ from Theorem \ref{theoprocessapp} is up to $\mu_{N^d,s,\varrho}$-equivalence the unique diffusion process having $\mu_{N^d,s,\varrho}$ as symmetrizing measure and solving the martingale problem for $\big(\mathcal{H}^{N^d,{s},\varrho},D(\mathcal{H}^{N^d,s,\varrho})\big)$, i.e., for all $G\in D(\mathcal{H}^{N^d,s,\varrho})$
\begin{align*}
\widetilde{G}(\mathbf{X}_t)-\widetilde{G}(\mathbf{X}_0)+\int_0^t \Big(\mathcal{H}^{N^d,{s},\varrho}G\Big)(\mathbf{X}_{\tau})\,d\tau,\quad t\ge 0,
\end{align*}
is an $\mathbf{F}_t$-martingale under $\mathbf{P}^{N^d,s,\varrho}_\phi$ (hence starting in $\phi$) for $\mathcal{E}^{N^d,s,\varrho}$-quasi all $\phi\in \overline{\Omega_{\scriptscriptstyle{d,N}}^+}$.
\end{theorem}
\begin{proof}
Use Remark \ref{remcondapp} and apply Theorem \ref{theomartingale}.
\end{proof}

\begin{corollary}\label{coroskoroapp}
Suppose that Condition \ref{condpotential} is satisfied. Let $x\in D_{\scriptscriptstyle{d,N}}$. We denote by $\pi_x:\overline{\Omega_{\scriptscriptstyle{d,N}}^+}\to[0,\infty)$, $\phi\mapsto \phi_x$, the projection on the $x$-th coordinate. The process $\mathbf{M}^{N^d,s,\varrho}$ is characterized via its coordinate processes $\big(\mathbf{X}_t^{x}\big)_{t\ge 0}:=\big(\pi_x(\mathbf{X}_{t})\big)_{t\ge 0}$ by
\begin{multline}\label{repsolapp}
\mathbf{X}_{t}^{x}-\mathbf{X}_0^{x}=\mathbbm{1}_{\Omega_{\scriptscriptstyle{d,N}}^+}\big(\mathbf{X}_{t}\big)\,\sqrt{2}\,B^x_t-\int_{0}^{t}\mathbb{V}'\big(x,\mathbf{X}_{\tau}\big)\mathbbm{1}_{\Omega_{\scriptscriptstyle{d,N}}^+}\big(\mathbf{X}_{\tau}\big)\,d\tau\\
+\sum_{\varnothing\not=B\subsetneq D_{\scriptscriptstyle{d,N}}}\left\{\begin{array}{ll}
  \mathbbm{1}_{\Omega_{\scriptscriptstyle{d,N,B}}^{+\scriptscriptstyle{\text{,wet}}}}\big(\mathbf{X}_{t}\big)\,\sqrt{2}\,B^x_t-\int_0^{t} \mathbb{V}'\big(x,\mathbf{X}_{\tau}\big)\mathbbm{1}_{\Omega_{\scriptscriptstyle{d,N,B}}^{+\scriptscriptstyle{\text{,wet}}}}\big(\mathbf{X}_{\tau}\big)\,d\tau, & \text{if }x\in B\\
  \frac{1}{s}\,\int_0^{t}\mathbbm{1}_{\Omega_{\scriptscriptstyle{d,N,B}}^{+\scriptscriptstyle{\text{,wet}}}}\big(\mathbf{X}_{\tau}\big)\,d\tau, & \text{if }x\in D_{\scriptscriptstyle{d,N}}\setminus B
  \end{array}\right.\\
+\frac{1}{s}\int_0^{t}\mathbbm{1}_{\{(0,\ldots,0)\}}\big(\mathbf{X}_{\tau}\big)\,d\tau,
\end{multline}
where $(B^x_t)_{t\ge 0}$ is a one dimensional Brownian motion with $B^x_0=0$,
\begin{align*}
\mathbb{V}'(x,\phi):=\sum_{\stackunder{|x-y|_{\text{euc}}=1}{y\in{{\overline{D_{\scriptscriptstyle{d,N}}}}}}}V'\big(\phi_x-\phi_y\big),\quad\phi\in \overline{\Omega_{\scriptscriptstyle{d,N}}^+},
\end{align*}
with pair interaction potential $V$.
\end{corollary}

\begin{proof}
Use Remark \ref{remcondapp} and apply Corollary \ref{coroskoro}.
\end{proof}

\begin{remark}\label{remskoroapp}
For $\varnothing\not=B\subsetneq D_{\scriptscriptstyle{d,N}}$ we have by Remark \ref{remBM} that on the boundary parts $\Omega_{\scriptscriptstyle{d,N,B}}^{+\scriptscriptstyle{\text{,wet}}}$ of $\partial \big(\overline{\Omega_{\scriptscriptstyle{d,N}}^+}\big)$ the constructed process is a $(\#B)$-dimensional Brownian motion scaled by $\sqrt{2}$ when it spends time there. 
\end{remark}

\begin{remark}
(\ref{repsolapp}) provides a weak solution to (\ref{sde}) in the sense of N.~Ikeda and Sh.~Watanabe (see e.g.~\cite[Chap.~2]{WaIk89}) for $\mathcal{E}^{N^d,s,\varrho}$-quasi every starting point in $\overline{\Omega_{\scriptscriptstyle{d,N}}^+}$.
\end{remark}
\subsection*{Acknowledgment}
We thank Benedikt Heinrich, Tobias Kuna, Michael R\"ockner and Heinrich von Weizs\"acker for discussions and helpful comments. Financial support through the DFG project GR 1809/8-1  is
gratefully acknowledged.

\end{document}